\documentclass[a4paper]{article}

\usepackage[utf8]{inputenc}
\usepackage[T2A]{fontenc}
\usepackage[english]{babel}

\usepackage{xcolor}
\usepackage{amsmath}
\usepackage{amsthm}
\usepackage{amsxtra}
\usepackage{amsfonts,amssymb}

\usepackage{dsfont}

\usepackage{young}
\usepackage{ytableau}

\newtheorem{Th}{Theorem}[section]
\newtheorem{Prop}[Th]{Proposition}
\newtheorem{Lm}[Th]{Lemma}
\newtheorem{Co}[Th]{Corollary}

\theoremstyle{definition}
\newtheorem{Def}[Th]{Definition}

\newtheorem{Rem}{Remark}[section]

\numberwithin{equation}{section}

\newcommand{\Tab}{\operatorname{Tab}}
\newcommand{\Aut}{\operatorname{Aut}}
\newcommand{\supp}{\operatorname{supp}}
\newcommand{\nat}{\operatorname{nat}}
\newcommand{\sgn}{\operatorname{sgn}}
\newcommand{\Tr}{\operatorname{Tr}}
\newcommand{\tr}{\operatorname{tr}}
\newcommand{\id}{\operatorname{id}}
\newcommand{\exchar}{\operatorname{ex\,Char}}
\newcommand{\Div}{\operatorname{Div}}

\date{}
\title{Indecomposable characters on direct limit of symmetric groups with diagonal embeddings}

\author{N.~Nessonov,
	\and
	N.~T.~S.~Ngo}

\begin{document}
	
	\maketitle
	
	\begin{abstract}
		In this paper we obtain the complete description of all indecomposable characters (central positive-definite functions) of inductive limits of the symmetric groups under block diagonal embedding.
		As a corollary we obtain the full classification of the isomorphism classes of these inductive limits.
		
		\medskip
		
		\emph{Keywords:} Infinite symmetric group; Character; Factor representation.
	\end{abstract}
	
	\section{Introduction}\label{intro}
	Consider the space $X=[0,1)$ with the standard Lebesgue measure $\nu$. Denote by $\Aut_0(X,\nu)$ the group of automorphisms of the space $(X,\nu)$ which preserve the measure $\nu$. In particular, one can consider finite subgroups of this group, which correspond to the so called \emph{rational rearrangements} of a $X$.
	
	Namely, each symmetric group $\mathfrak{S}_{N}$ regarded as group of bijections of the set $\mathbb{X}_N=\{0,1,2,\ldots,N-1\}$ can be embedded into $\Aut_0(X,\mu)$ in the following way. For each $\sigma\in\mathfrak{S}_{N}$ define automorphism
	$T_N(\sigma)\in\Aut_0(X,\mu)$ via the formula
	\begin{equation*}
		T_N(\sigma)(x)=\frac{Nx-[Nx]+\sigma([Nx])}{N},\;x\in[0,1).
	\end{equation*}
	In other words, the map $T_N(\sigma)$ acts on half-closed intervals $[\frac{k-1}{N},\frac{k}{N})$, $k\in\overline{1,N}$\    \ \footnote{Here we use the notation $\overline{p,q}$ for the set $\left\{p+1,p+2,\ldots q \right\}\subset \mathbb{Z}$.} via the permutation $\sigma$. It is easy to verify that $T_N$ is an injective homomorphism.
	
	In order to understand which automorphism in $T_{n_1n_2}\left(\mathfrak{S}_{n_1n_2} \right)$ coincides with $T_{n_1}(\sigma)$, let us represent each element $y$ of $\mathbb{X}_{n_1n_2}$ as $y=x_1+n_1x_2$, where $x_1\in\mathbb{X}_{n_1}$ and $x_2\in\mathbb{X}_{n_2}$. Then
	\begin{eqnarray*}
		T_{n_1}(\sigma)\left[\frac{k}{n_1}+\frac{i}{n_1n_2},\frac{k}{n_1}+\frac{i+1}{n_1n_2}\right)=
		\left[\frac{\sigma(k)}{n_1}+\frac{i}{n_1n_2},\frac{\sigma(k)}{n_1}+\frac{i+1}{n_1n_2}\right),
	\end{eqnarray*}
	where $k\in\overline{0,n_1-1}$, $i\in\overline{0,n_2-1}$. It means that $T_{n_1}(\sigma)=T_{n_1n_2}(\mathfrak{i}(\sigma))$ where the permutation $\mathfrak{i}(\sigma)\in\mathfrak{S}_{n_1n_2}$ acts as follows
	\begin{eqnarray*}
		\mathfrak{i}(\sigma)(x_1+n_1x_2)=\sigma(x_1)+n_1x_2,~\text{where}~ x_1\in\mathbb{X}_{n_1}, x_2\in\mathbb{X}_{n_2}.
	\end{eqnarray*}
	In this way we obtain a natural embedding of the symmetric group $\mathfrak{S}_{n_1}$ into $\mathfrak{S}_{n_1n_2}$ which corresponds to the inclusion $T_{n_1}(\mathfrak{S}_{n_1})\subset T_{n_1n_2}(\mathfrak{S}_{n_1n_2})$.
	
	If we identify $\mathbb{X}_{n_1n_2}$ with $\mathbb{X}_{n_1}\times\mathbb{X}_{n_2}$ using the correspondence $(x,y)\leftrightarrow x+n_1y$, then  $\mathfrak{i}(\sigma)$ acts as follows:
	\begin{equation*}
		\mathfrak{i}(\sigma)((x,y))=(\sigma(x),y),~x\in\mathbb{X}_{n_1}, y\in\mathbb{X}_{n_2}.
	\end{equation*}	
	Continuing this process and letting $N_k=n_1n_2\cdots n_k$ we obtain the infinite chain of subgroups
	\begin{equation*}
		T_{N_1}\left( \mathfrak{S}_{N_1} \right)\subset T_{N_2}\left( \mathfrak{S}_{N_2} \right)\subset \ldots\subset T_{N_k}\left( \mathfrak{S}_{N_k} \right)\subset\ldots\subset \Aut_0(X,\nu).
	\end{equation*}
	Their union $\bigcup\limits_k T_{N_k}\left(\mathfrak{S}_{N_k}\right)$ is a countable subgroup in $\Aut_0(X,\nu)$. This subgroup of $\Aut_0(X,\nu)$ is naturally isomorphic to the inductive limit $\varinjlim\mathfrak{S}_{N_k}$ which corresponds to the embedding $\mathfrak{i}(\sigma)$. Note that, in general, different sequences $\left\{N_k \right\}$ define non-isomorphic inductive limits.
	In particular, they might be simple as well as contain a nontrivial normal subgroup. In the present paper we obtain the complete description of pairs of sequences $\left\{N_k'\right\}$ and $\left\{N_k''\right\}$ for which the corresponding inductive limits are isomorphic (see Theorem \ref{isomorphclass}).
	
	For each prime number $p$ denote by $\deg_p(N_k)$ the degree of $p$ in the prime factorization $N_k=\prod_p p^{\deg_p(N_k)}$. In case when for each prime $p$ the sequence $\left\{\deg_p(N_k)\right\}_{k=1}^\infty$ is unbounded, i.e. $\lim\limits_{k\to\infty}\deg_p(N_k)=\infty$, the group $\bigcup\limits_k T_{N_k}\left(\mathfrak{S}_{N_k}\right)$ is called \emph{the group of rational rearrangements of a segment}. Denote this group as $\mathfrak{S}_\mathbb{Q}$. In particular, $\mathfrak{S}_\mathbb{Q}$ is a simple group.
	In \cite{Goryachko} the full description of {\it indecomposable characters} on $\mathfrak{S}_\mathbb{Q}$ was obtained. Recall that positive definite function $\chi$ on group $G$ is called \emph{central} or \emph{character} if it satisfies the following condition
	\begin{eqnarray*}
		\chi(gh)=\chi(hg)~\text{for all}~g,h\in G.
	\end{eqnarray*}
	In the present paper we consider only {\it normalized } characters, i.e. which equal to 1 on the identity element of $G$.
	A character $\chi$ is called \emph{indecomposable} if the unitary representation $\Pi_\chi$ of the group $G$, constructed via $\chi$ according to the Gelfand-Naimark-Segal (GNS) construction, is a factor representation. Namely, in this case the operators $\Pi_\chi(G)$ generate a factor of type ${\rm II}_1$ \cite{Takesaki_1}. This definition is equivalent to the following property: indecomposable characters are the extreme points of the simplex of all characters.
	
	An important special case of the group $\varinjlim\mathfrak{S}_{N_k}$, where $N_k=2^{M_k}$, was studied by A. Dudko \cite{Dudko}. This group is also a simple group and in \cite{Dudko} all indecomposable characters of this group were found. In  \cite{Dudko_MED} the full description of indecomposable characters was given for more general symmetric groups which act on the paths of the Bratelli's diagram. However, the results of papers \cite{Dudko}, \cite{Dudko_MED} and \cite{Goryachko} did not cover the case of an arbitrary sequence $\left\{N_k\right\}$. In the present paper we obtain the description of all characters on groups $\varinjlim\mathfrak{S}_{N_k}$ without any additional conditions on the sequences  $\left\{N_k\right\}$.
	
	\subsection{The inductive limit of symmetric groups}\label{group_def}
	In this subsection we define the group $\mathfrak{S}_{\widehat{\mathbf{n}}}$ as an inductive limit of symmetric groups with the diagonal embedding.
	
	We regard the group $\mathfrak{S}_N$ as the group of all bijections (symmetries) of the set $\mathbb{X}_N=\{0,1,2,\ldots,N-1\}$. We identify $\mathbb{X}_{NM}$ with $\mathbb{X}_{N}\times\mathbb{X}_{M}$ via the isomorphism $\mathbb{X}_{NM}\ni x_1+Nx_2\leftrightarrow (x_1,x_2)\in\mathbb{X}_{N}\times\mathbb{X}_{M}$. Denote by $\mathfrak{i}$ the embedding $\mathfrak{S}_{N}$ into $\mathfrak{S}_{NM}$ defined as follows
	\begin{eqnarray}\label{embedding}
		\mathfrak{i}(\sigma)(x_1,x_2)=(\sigma(x_1),x_2).
	\end{eqnarray}
	In order to define formally the inductive limit of groups $\mathfrak{S}_N$ which correspond to the embedding \eqref{embedding}, consider the sequence $\widehat{\mathbf{n}}=\left\{n_k \right\}_{k=1}^\infty$ of positive integers, where $n_k>1$ for all $k$.
	
	Suppose that $j>k$. Put $N_k=\prod_{i=1}^{k}n_i$, $^j\!\!N_k=\frac{N_j}{N_k}$. Then $\mathbb{X}_{N_j}=\mathbb{X}_{N_k}\otimes\mathbb{X}_{(\!^j\!\!N_k)}$. Denote by $\mathfrak{i}_{k,j}$ the embedding of the group $\mathfrak{S}_{N_k}$ into $\mathfrak{S}_{N_j}$, defined as in \eqref{embedding}.
	Define the group $\mathfrak{S}_{\widehat{\mathbf{n}}}=\bigcup\limits_{k=1}^\infty\mathfrak{S}_{N_k}$ by identifying elements $\mathfrak{S}_{N_k}$ with their images under the maps $\mathfrak{i}_{k,j}$. In other words, $\mathfrak{S}_{\widehat{\mathbf{n}}}$ is the inductive limit $\varinjlim\mathfrak{S}_{N_k}$ that corresponds to the embeddings $\mathfrak{i}_{k,j}$. Denote by $\id$ the identity element of the group $\mathfrak{S}_{\widehat{\mathbf{n}}}$.
	
	\begin{Rem}\label{aut_subgroup}
		The group $\mathfrak{S}_{\widehat{\mathbf{n}}}$ is isomorphic to the subgroup $\bigcup\limits_k T_{N_k}\left(\mathfrak{S}_{N_k}\right)$ of the group $\Aut_0(X,\nu)$.
	\end{Rem}
	
	Suppose that $\sigma\in\mathfrak{S}_{N_k}$ belongs to the conjugacy class $\mathfrak{C}_{\,^k\!\mathfrak{m}}$ in $\mathfrak{S}_{N_k}$ consisting of permutations of the cycle type $\,^k\!\mathfrak{m}=(\,^k\!m_1, \,^k\!m_2,\ldots, \,^k\!m_l)$, where $\,^k\!m_i$  is the number of cycles of the length $i$ in the decomposition of $\sigma$ into disjoint cycles (see definitions in Subsection \ref{minimal_element}). Then  $\mathfrak{i}_{k,j}\left(\mathfrak{C}_{\,^k\!\mathfrak{m}}\right)\subset \mathfrak{C}_{\,^j\!\mathfrak{m}}$, where \begin{eqnarray}\label{cycle_type_evolution}
		\,^j\!\mathfrak{m}=\left(\,^j\!m_1,\,^j\!m_2,\ldots,\,^j\!m_l  \right)=\left(  \,^k\!m_1\frac{N_j}{N_k},  \,^k\!m_2\frac{N_j}{N_k},\ldots,  \,^k\!m_l\frac{N_j}{N_k} \right).
	\end{eqnarray}
	
	For each $\sigma\in\mathfrak{S}_{N_k}$ define $\supp_{_{N_k}}\,\sigma=\left\{x\in\mathbb{X}_{N_k}:\sigma x\neq x \right\}$.
	Note that if $\sigma\in \mathfrak{S}_{N_k}$, then
	\begin{eqnarray}\label{coherence_supp}
		\#\left(\supp_{_{N_j}}\,\mathfrak{i}_{k,j}(\sigma)\right)=\#\left(\supp_{_{N_k}}\,\sigma\right)\cdot \frac{N_j}{N_k}.
	\end{eqnarray}
	
	Define the multiplicative character $\sgn_{N_k}: \mathfrak{S}_{N_k}\rightarrow\{-1,1\}$ via the formula $\sgn_{N_k}\,(\sigma)=(-1)^{\kappa_{_{N_k}}(\sigma)}$. Here by $\kappa_{_{N_k}}(\sigma)$ we denote the minimal number of factors in the decomposition of $\sigma$ into the product of transpositions. It is known that if $\sigma\in\mathfrak{C}_{\,^k\!\mathfrak{m}}$, then $\kappa_{_{N_k}}(\sigma)=N_k-\sum_{p=1}^l\,^k\!m_p$. Note that
	\begin{eqnarray*}
		\kappa_{_{N_j}}(\sigma)=\frac{N_j}{N_k} \kappa_{_{N_k}}(\sigma).
	\end{eqnarray*}
	
	Therefore, $\sgn_{N_j}\,(\sigma)=\left(\sgn_{N_k}\,(\sigma)\right)^{N_j/N_k}$. This implies that for each $s\in \mathfrak{S}_{\widehat{\mathbf{n}}}$ there exists $M(s)$ such that
	\begin{eqnarray*}
		\sgn_{N_j}(s)=\sgn_{N_i}(s) ~\text{for all}~i,j>M(s).
	\end{eqnarray*}
	Thus, there exists a limit $\sgn_\infty(s)=\lim\limits_{j\to\infty}\sgn_{N_j}(s)$. The function $\sgn_\infty$ is a multiplicative character on the group $\mathfrak{S}_{\widehat{\mathbf{n}}}$.
	
	
	The following statement is immediate.
	\begin{Prop}\label{simple_subgroup}
		Denote by $\mathfrak{A}_{\widehat{\mathbf{n}}}$ the subgroup $\left\{g\in\mathfrak{S}_{\widehat{\mathbf{n}}}: \sgn_\infty(g)=1\right\}$. Then
		\begin{itemize}
			\item[{\rm (a)}] $\mathfrak{A}_{\widehat{\mathbf{n}}}$ is a simple group;
			\item[{\rm (b)}] if the sequence $\widehat{\mathbf{n}}=\left\{n_k \right\}_{k=1}^\infty$ contains infinitely many even numbers, then $\mathfrak{S}_{\widehat{\mathbf{n}}}$ is a simple group.
		\end{itemize}
	\end{Prop}
	
	\subsection{The main result}
	
	Let $\mathbb{M}_N(\mathbb{C})$ be the algebra of the complex $N\times N$ matrices, and let $I_N$ be the identity $N\times N$ matrix. Denote by $\Tr_N$ an ordinary trace on $\mathbb{M}_N(\mathbb{C})$. Set $\tr_N(A)=\frac{1}{N}\Tr_N(A)$, where $A\in\mathbb{M}_N(\mathbb{C})$. Define on $\mathbb{M}_N(\mathbb{C})$ an inner product $\left< A,B\right>_N=\tr_N(B^*A)$,  $A,B\in\mathbb{M}_N(\mathbb{C})$. The elements of $\sigma\in\mathfrak{S}_N$ are realized as $\{0,1\}$-matrices $\Sigma_\sigma=\left[\delta_{i\sigma(j)} \right]$ in $\mathbb{M}_N(\mathbb{C})$, where $\delta_{ij}=\left\{
	\begin{array}{rl}
		1, \text{ if } i=j,
		\\
		0,\text{ if } i\neq j.
	\end{array}
	\right.$. The operators of the left multiplication by $\Sigma_\sigma$ define on $\mathbb{M}_N(\mathbb{C})$ a unitary representation $\mathfrak{L}_N$ of the group $\mathfrak{S}_N$:
	\begin{equation*}
		\mathfrak{L}(\sigma)A=\Sigma_\sigma\cdot A,~A\in\mathbb{M}_N(\mathbb{C}).
	\end{equation*}
	Put $\varphi_N(\sigma)=\left<\mathfrak{L}(\sigma)I_N,I_N \right>=\tr_N(\Sigma_\sigma)$. Clearly, $\varphi_N$ is a character on $\mathfrak{S}_N$ and
	\begin{eqnarray}
		\varphi_N(\sigma)=1-\frac{\#\supp_N\,\sigma}{N}.
	\end{eqnarray}
	Note that if $\sigma\in\mathfrak{S}_{N_k}$, then for $j>k$ from \eqref{coherence_supp} we have
	\begin{eqnarray*}
		\varphi_{N_j}\left(\,\mathfrak{i}_{k,j}(\sigma) \right)= \varphi_{N_k}\left(\sigma \right).
	\end{eqnarray*}
	Hence, the sequence $\left\{\varphi_{N_k}\right\}$ defines a character $\chi_{\nat}$ on the inductive limit $\mathfrak{S}_{\widehat{\mathbf{n}}}=\varinjlim\mathfrak{S}_{N_k}$. In other words, for $\sigma\in\mathfrak{S}_{N_k}\subset \mathfrak{S}_{\widehat{\mathbf{n}}}$ we have
	\begin{equation}\label{chi_nat}
		\chi_{\nat}(\sigma)=\varphi_{N_k}(\sigma)=1-\frac{\#\supp_{N_k}\,\sigma}{N_k}.
	\end{equation}
	
	Our main result is the following theorem.
	
	\begin{Th}\label{main_theorem}
		Let the character $\chi^p_{\nat}$, where $p\in \mathbb{N}\cup\{0,\infty\}$, be defined as $\chi_{\nat}^p(\sigma)=\left(\chi_{\nat}(\sigma)\right)^p$ when $p\in  \mathbb{N}\cup\{0\}$ and $\chi_{\nat}^\infty(\sigma)=\left\{
		\begin{array}{rl}
			1, \text{ if } \sigma=\id,\\
			0, \text{ if } \sigma\neq\id.
		\end{array}
		\right.$
		If $\chi$ is an indecomposable character on $\mathfrak{S}_{\widehat{\mathbf{n}}}$, then there exists $p\in \mathbb{N}\cup\{0,\infty\}$ such that $\chi=\chi^p_{\nat}$ or $\chi=\sgn_\infty\cdot\chi^p_{\nat}$, where
		$\left(\sgn_\infty\cdot\chi^p_{\nat}\right)(\sigma)=\sgn_\infty(\sigma)\cdot\chi^p_{\nat}(\sigma)$, $\sigma\in\mathfrak{S}_{\widehat{\mathbf{n}}}$.
	\end{Th}
	
	As a corollary, we also obtain the complete classification of the isomorphism classes of groups $\mathfrak{S}_{\widehat{\mathbf{n}}}$.
	
	\begin{Th}\label{isomorphclass}
		Let $\widehat{\mathbf{n}}'=\{n_k'\}_{k=1}^{\infty}$ and $\widehat{\mathbf{n}}''=\{n_k''\}_{k=1}^{\infty}$, where $n_k',n_k''>1$ for all $k$, be the sequences of positive integers. Put $N_k'=\prod_{i=1}^{k}n_i'$ and $N_k''=\prod_{i=1}^{k}n_i''$. Then, groups $\mathfrak{S}_{\widehat{\mathbf{n}}'}$ and $\mathfrak{S}_{\widehat{\mathbf{n}}''}$ (see Section \ref{group_def}) are isomorphic iff for each prime number $p$ the following condition holds:
		\begin{equation}\label{isomorph_condition}
			\lim_{k\to\infty}\deg_p(N_k')=\lim_{k\to\infty}\deg_p(N_k'').
		\end{equation}
		In other words, groups $\mathfrak{S}_{\widehat{\mathbf{n}}'}$ and $\mathfrak{S}_{\widehat{\mathbf{n}}''}$ are isomorphic iff for each prime $p$ either both sequences $\{\deg_p(N_k')\}_{k=1}^{\infty}$ and $\{\deg_p(N_k'')\}_{k=1}^{\infty}$ are unbounded, or there is a non-negative integer $d_p$ such that $\deg_p(N_k')=\deg_p(N_k'')=d_p$ for all sufficiently large $k$.
	\end{Th}
	\subsection{${\rm II}_1$-factor-representations of the group $\mathfrak{S}_{\widehat{\mathbf{n}}}$ and spherical representations of $\mathfrak{S}_{\widehat{\mathbf{n}}}\times \mathfrak{S}_{\widehat{\mathbf{n}}}$.}
Let $\chi$ be an arbitrary character on $\mathfrak{S}_{\widehat{\mathbf{n}}}$. Consider GNS-representation $(\pi_\chi,\mathcal{H}_\chi,\xi_\chi)$ of $\mathfrak{S}_{\widehat{\mathbf{n}}}$ corresponding  to $\chi$, where $\xi_\chi$ is the unit cyclic vector for $\pi_\chi\left( \mathfrak{S}_{\widehat{\mathbf{n}}}\right)$ in the Hilbert space $\mathcal{H}_\chi$ such that $\chi(\sigma)=\left<\pi_\chi(\sigma)\xi_\chi,\xi_\chi\right>$ for all $\sigma\in \mathfrak{S}_{\widehat{\mathbf{n}}}$.
Denote by $M$ $w^*$-algebra generated by a set of operators $\pi_\chi\left(\mathfrak{S}_{\widehat{\mathbf{n}}}\right)$.	Character $\chi$ defines a finite, normal faithful trace ${\rm tr}$ on $M$. Namely,
\begin{eqnarray*}
{\rm tr}\left(\pi_\chi(\sigma)\right)=\chi(\sigma), \sigma\in \mathfrak{S}_{\widehat{\mathbf{n}}}.
\end{eqnarray*}
Taking into account the definition of GNS-construction, we assume that $\mathcal{H}_\chi=L^2(M,{\rm tr})$, where inner product $\left<\cdot,\cdot \right>$ is defined as follows
\begin{eqnarray*}
\left<m_1,m_2 \right>={\rm tr}\left(m_2^*m_1\right),\;\;m_1,m_2\in M,
\end{eqnarray*}
and $\xi_\chi$ is an identity operator from $M$. Finally, we recall that the operators $\pi_\chi(\sigma)$, $\sigma\in  \mathfrak{S}_{\widehat{\mathbf{n}}}$ act on $L^2(M,{\rm tr})$ by left multiplication
\begin{eqnarray}\label{left_component}
L^2(M,{\rm tr})\ni v\stackrel{\pi_\chi(\sigma)}{\rightarrow} \pi_\chi(\sigma)\cdot v\in L^2(M,{\rm tr}).
\end{eqnarray}
Denote by $\mathcal{B}\left(\mathcal{H}_\chi\right)$ the set of all bounded linear operators on $\mathcal{H}_\chi$, and put $$M'=\left\{A\in\mathcal{B}\left(\mathcal{H}_\chi \right): AB=BA \text{ for all } B\in M\right\}.$$
Since ${\rm tr}$ is a central state on $M$; i. e. ${\rm tr}(m_1m_2)={\rm tr}(m_2m_1)$ for all $m_1,m_2\in M$, the mapping
\begin{eqnarray}\label{right_component_general_case}
L^2(M,{\rm tr})\ni v\stackrel{\pi'_\chi(\sigma)}{\rightarrow} v\cdot \pi_\chi(\sigma^{-1})\in L^2(M,{\rm tr})
\end{eqnarray}
define an unitary operator   $\pi'_\chi(\sigma)\in M'$ on $\mathcal{H}_\chi$.
Denote by $J$ the antilinear isometry which acts as follows: $L^2(M,{\rm tr})\ni m\stackrel{J}{\rightarrow}m^*$. It is clear that $J^2=I$.

It follows from the above that
\begin{eqnarray}\label{JMJ in M'}
JMJ\subset M' \text{  and  } M\subset JM'J.
\end{eqnarray}
Let us prove that
\begin{eqnarray}\label{M=JMJ'}
 M= JM'J.
\end{eqnarray}
Let operator $A$ belongs to $JM'J$. Since $\xi_\chi$ is cyclic for $M$; i. e. a set $M\xi_\chi$ is norm dense in $L^2\left(M,{\rm tr} \right)$, there exists a sequence $\left\{A_n \right\}_{n\in\mathbb{N}}\subset M$ such that
\begin{eqnarray}\label{approx_A}
\lim\limits_{n\to \infty}\left\|A\xi_\chi-A_n\xi_\chi\right\|_{L^2(M,{\rm tr})}=0.
\end{eqnarray}
Hence, applying centrality of $\chi$, we have
\begin{eqnarray*}\label{fundamental_sequence}
\lim\limits_{m,n\to \infty}\left\|A_m^*\xi_\chi-A_n^*\xi_\chi\right\|_{L^2(M,{\rm tr})}=0.
\end{eqnarray*}
It follows from this that sequence $\left\{A^*_n\xi_\chi \right\}$ converges in norm to $\eta\in L^2(M,{\rm tr})$. Hence for each $U'\in JMJ$ we  obtain  next chain of equalities:
\begin{eqnarray*}
\left<(U')^*\xi_\chi,\eta\right>=\lim\limits_{n\to\infty}\left<U')^*\xi_\chi,A_n^*\xi_\chi\right>\\
=\lim\limits_{n\to\infty}\left< A_n\xi_\chi,U'\xi_\chi\right>\stackrel{(\ref{approx_A})}{=}\left< A\xi_\chi,U'\xi_\chi\right>=\left< (U')^*\xi_\chi,A^*\xi_\chi\right>.
\end{eqnarray*}
Therefore, using the cyclicity of $\xi_\chi$ for $JMJ$, we have
\begin{eqnarray}\label{adjoint_convergence}
A^*\xi_\chi=\eta; ~\text{ i.e. }~ \lim\limits_{n\to\infty}\left\|A^*_n\xi_\chi -A^*\xi_\chi\right\|_{L^2(M,{\rm tr})}=0.
\end{eqnarray}
\begin{Lm}
Put $\omega(x)=\left<x\xi_\chi,\xi_\chi\right>$, where $x\in\mathcal{B}\left(\mathcal{H}_\eta\right)$. Then
 \begin{itemize}
   \item {\bf 1.} $omega$ is a central state on $M'$ and on $JM'J$; i. e.
\begin{eqnarray}
\begin{split}
\omega \left(A'B'\right)=\omega \left(B'A'\right)  \text{ for all } A',B'\in M',
\text{ and }\\
\omega \left(AB\right)=\omega \left(BA\right)  \text{ for all } A,B\in JM'J;
\end{split}
\end{eqnarray}
   \item {\bf 2.} $AB'=B'A$ for all $A\in JM'J$ and $B'\in M'$.
 \end{itemize}
 \end{Lm}
\begin{proof}
{\bf Property 1.}
Since $\xi_\chi$ is cyclic vector for $M$, there exists the sequences $\left\{A_n \right\}$, $\left\{B_n \right\}$ in $M$ such that
\begin{eqnarray*}
\lim\limits_{n\to\infty}\left\|A\xi_\eta-A_n\xi_\eta \right\|_{L^2(M,{\rm tr})}=\lim\limits_{n\to\infty}\left\|B\xi_\eta-B_n\xi_\eta \right\|_{L^2(M,{\rm tr})}=0.
\end{eqnarray*}
Hence, using (\ref{adjoint_convergence}), we obtain
\begin{eqnarray*}
\lim\limits_{n\to\infty}\left\|A^*\xi_\eta-A_n^*\xi_\eta \right\|_{L^2(M,{\rm tr})}=\lim\limits_{n\to\infty}\left\|B^*\xi_\eta-B_n^*\xi_\eta \right\|_{L^2(M,{\rm tr})}=0.
\end{eqnarray*}
Therefore, $\omega\left(AB\right)=\lim\limits_{n\to\infty}\omega\left(A_nB_n\right)=\lim\limits_{n\to\infty}{\rm tr}\,\left(A_nB_n\right)=\lim\limits_{n\to\infty}{\rm tr}\,\left(B_nA_n\right)$ $=\lim\limits_{n\to\infty}\omega\left(B_nA_n\right)=\omega\left(BA\right)$. We leave it to the reader to verify that  $\omega \left(A'B'\right)$ $=\omega \left(B'A'\right)  \text{ for all } A',B'\in M'$.

{\bf Property 2.} Since $J^2=I$, then, by cyclicity of the vector $\xi_\chi$ for $M$,  there exists the sequences $\left\{A_n \right\}$, $\left\{B_n \right\}$ in $M$ such that
\begin{eqnarray*}
\lim\limits_{n\to\infty}\left\|A\xi_\chi-A_n\xi_\chi\right\|=0 \text{ and  } \lim\limits_{n\to\infty}\left\|JB'J\xi_\chi-B_n\xi_\chi\right\|=0.
\end{eqnarray*}
Hence, using the equality $\left\|B'\xi_\chi-JB_n J \xi_\chi \right\|=\left\|(B')^*\xi_\chi-(JB_n J)^* \xi_\chi \right\|$, which is due to the fact that $\omega$ is central state on $M'$, we obtain for any $U'\in JMJ$ and $V\in M$
\begin{eqnarray*}
\left<AB'U'\xi_\chi,V\xi_\chi\right>=\lim\limits_{n\to\infty}\left<A_n\,JB_nJ\,U'\xi_\chi,V\xi_\chi\right>\stackrel{(\ref{JMJ in M'})}{=}\lim\limits_{n\to\infty}\left<JB_nJ\,A_n\,U'\xi_\chi,V\xi_\chi\right>\\
=\lim\limits_{n\to\infty}\left<A_n\,U'\xi_\chi,V\left(JB_nJ\right)^*\xi_\chi\right>=\left<A\,U'\xi_\chi,V(B')^*\xi_\chi\right>=
\left<B'A\,U'\xi_\chi,V\xi_\chi\right>.
\end{eqnarray*}
Consequently, $AB'=B'A$. In particular, this establishes equality  (\ref{M=JMJ'}).
\end{proof}

Now we define the representation $\pi_\chi^{(2)}$ of the group $\mathfrak{S}_{\widehat{\mathbf{n}}}\times \mathfrak{S}_{\widehat{\mathbf{n}}}$ as follows
\begin{eqnarray}
\pi_\chi^{(2)}((\sigma_1,\sigma_2))=\pi_\chi(\sigma_1)\pi_\chi'(\sigma_2),\; (\sigma_1,\sigma_2)\in \mathfrak{S}_{\widehat{\mathbf{n}}}\times \mathfrak{S}_{\widehat{\mathbf{n}}}.
\end{eqnarray}
Hence, applying (\ref{left_component}) and (\ref{right_component_general_case}), we obtain
\begin{eqnarray}\label{cyclic vector is diagonal fixed}
\pi_\chi^{(2)}((\sigma,\sigma))\xi_\eta=\xi_\eta~ \text{for all }~ \sigma\in  \mathfrak{S}_{\widehat{\mathbf{n}}}.
\end{eqnarray}

Denote by $\mathcal{C}(M)$ center of $w^*$-algebra $M$.
\begin{Prop}\label{center=commutant}
$\left(\pi_\chi^{(2)}\left(\mathfrak{S}_{\widehat{\mathbf{n}}}\times \mathfrak{S}_{\widehat{\mathbf{n}}}\right)\right)'=\mathcal{C}(M)$.
\end{Prop}
\begin{proof}
It follows from (\ref{right_component_general_case}) and (\ref{M=JMJ'}) that $\left(\pi_\chi^{(2)}\left(\mathfrak{S}_{\widehat{\mathbf{n}}}\times \mathfrak{S}_{\widehat{\mathbf{n}}}\right)\right)'= M\cap M'=\mathcal{C}(M)$.
\end{proof}
Define an orthogonal projection  $E_k$ by
\begin{eqnarray}\label{fixed_diagonal}
E_k=\frac{1}{N_k!}\sum\limits_{\sigma\in \mathfrak{S}_{\widehat{\mathbf{n}}}}\pi_\chi^{(2)}\left(\sigma,\sigma\right).
\end{eqnarray}
It is clear that $E_k\geq E_{k+1}$. Therefore, there exists the limit $E=\lim\limits_{k\to\infty}E_k$. It follows from (\ref{fixed_diagonal}) that
\begin{eqnarray}
E\mathcal{H}_\chi=\left\{\eta\in\mathcal{H}_\chi:\pi_\chi^{(2)}((\sigma,\sigma))\eta=\eta~ \text{ for all }~ \sigma\in\mathfrak{S}_{\widehat{\mathbf{n}}}\right\}.
\end{eqnarray}

\begin{Prop}
The following three conditions are equivalent:
\begin{itemize}
  \item {\rm(i)} $\pi_\chi$ is a factor-representation;
  \item {\rm(ii)} representation $\pi_\chi^{(2)}$ is irreducible;
  \item {\rm(iii)} ${\rm dim}\,E\mathcal{H}_\chi=1$.
\end{itemize}
\end{Prop}
\begin{proof}
The equivalence {\rm(i)} and {\rm(ii)} follows from proposition \ref{center=commutant}.

Let us prove that {\rm(i)} $\Rightarrow$ {\rm(iii)}. On the contrary, suppose that ${\rm dim}\,E\mathcal{H}_\chi\geq 2$. Then there exists  unit vector $v\in E\mathcal{H}_\chi$ such that
\begin{eqnarray}\label{orthogonality}
\left<v,\xi_\chi\right>=0.
\end{eqnarray}
Since $\xi_\chi$  is cyclic vector for $M$, then
\begin{eqnarray}
\left\|v-m\xi_\chi\right\|_{\mathcal{H}_\chi}<1 ~\text{ for some }~ m\in M.
\end{eqnarray}
Hence, using (\ref{cyclic vector is diagonal fixed}) and applying an equality $$\pi_\chi^{(2)}((\sigma,\sigma))\,m\, \pi_\chi^{(2)}((\sigma^{-1},\sigma^{-1}))=\pi_\chi(\sigma)\, m \pi_\chi(\sigma^{-1}),$$ we have
\begin{eqnarray}
\left\|v- \pi_\chi(\sigma)\, m \pi_\chi(\sigma^{-1})\;\xi_\eta \right\|_\mathcal{H_\chi}<1 ~\text{ for all }~\sigma\in\mathfrak{S}_{\widehat{\mathbf{n}}}.
\end{eqnarray}
Consequently,
\begin{eqnarray}\label{approx_v}
\left\|v-\frac{1}{N_k!}\sum\limits_{\sigma\in\mathfrak{S}_{N_k}} \pi_\chi(\sigma)\, m \pi_\chi(\sigma^{-1})\;\xi_\eta \right\|_\mathcal{H_\chi}<1.
\end{eqnarray} It easy to check that sequence $\left\{m_k=\frac{1}{N_k!}\sum\limits_{\sigma\in\mathfrak{S}_{N_k}} \pi_\chi(\sigma)\, m \pi_\chi(\sigma^{-1}) \right\}\subset M$ converges in strong operator topology to an operator  $Em\in \mathcal{C}(M)$. Emphasize that we identify here $Em\in \mathcal{H}_\chi=L^2(M,{\rm tr})$ with the corresponding  left multiplication operator from $M$. By (\ref{approx_v}),
\begin{eqnarray*}
\left\|v-Em\xi_\chi \right\|_{\mathcal{H}_\chi}<1.
\end{eqnarray*}
Hence, using property (i), we obtain
\begin{eqnarray*}
\left\|v-\alpha\xi_\chi \right\|_\mathcal{H}=\sqrt{1+|\alpha|^2-2\Re(\alpha<v,\xi_\chi>)}\stackrel{(\ref{orthogonality})}{=}\sqrt{1+|\alpha|^2}<1,~ \text{ where }~ \alpha\in \mathbb{C}.
\end{eqnarray*}
To prove that (iii) implies (i) suppose contrary, that there exist an orthogonal projection $Q\in \mathcal{C}(M)$ with the properties:
\begin{eqnarray*}
v=Q\xi_\chi\neq 0,w=(I-Q)\xi_\chi\neq0.
\end{eqnarray*}
Since  $v,w$ are mutually orthogonal vectors from $E\mathcal{H}_\chi$, then ${\rm dim}\,E\mathcal{H}_\chi\geq2$.
\end{proof}
	\section{The realizations  of ${\rm II}_1$-representations}
	
	In this section we give the explicit construction of a type ${\rm II}_1$ factor representation of the group $\mathfrak{S}_{\widehat{\mathbf{n}}}$ and the corresponding irreducible representation of the group $\overline{\mathfrak{S}_{\widehat{\mathbf{n}}}}\times \overline{\mathfrak{S}_{\widehat{\mathbf{n}}}}$.
	
	\subsection{Preliminaries}\label{prelim_1}
	
	Denote by $\nu_m$ the uniform probability measure on the set $\mathbb{X}_m=\{0,1,\ldots, m-1\}$, i.e. $\nu_m(\{j\})=\frac{1}{m}$ for all $j\in\mathbb{X}_{m}$. Let $\mathbb{X}_{\widehat{\mathbf{n}}}=\prod\limits_{k=1}^\infty \mathbb{X}_{n_k}$. For $x=(x_1,x_2,\ldots)\in\mathbb{X}_{\widehat{\mathbf{n}}}$ we set $\,^j\!x=(x_1,x_2,\ldots,x_j)\in\prod\limits_{k=1}^j\mathbb{X}_{n_k}$. Each element $y\in \prod\limits_{k=1}^j\mathbb{X}_{n_k}$ defines a cylindric set
	\begin{eqnarray}\label{cylinder}
		\,^j\!\!\mathbb{A}_y=\left\{ x\in\mathbb{X}_{\widehat{\mathbf{n}}}:\,^j\!x=y \right\}\subset\mathbb{X}_{\widehat{\mathbf{n}}}.
	\end{eqnarray}
	Now introduce the probability measure $\nu_{\widehat{\mathbf{n}}}=\prod\limits_{k=1}^\infty \nu_{n_k}$ on  $\mathbb{X}_{\widehat{\mathbf{n}}}$ by the formula
	\begin{equation*}
		\nu_{\widehat{\mathbf{n}}}(\,^j\!\!\mathbb{A}_y)=\frac{1}{n_1n_2\cdots n_j}=\frac{1}{N_j}.
	\end{equation*}
	Let $\Aut_0\left( \mathbb{X}_{\widehat{\mathbf{n}}},\nu_{\widehat{\mathbf{n}}} \right)$ be the group of automorphisms of the Lebesgue space $\left(\mathbb{X}_{\widehat{\mathbf{n}}},\nu_{\widehat{\mathbf{n}}}  \right)$ which preserve the measure $\nu_{\widehat{\mathbf{n}}}$. It follows from the definition of $\mathfrak{S}_{\widehat{\mathbf{n}}}$ that
	\begin{equation*}
		\mathfrak{S}_{\widehat{\mathbf{n}}}\subset \Aut_0\left( \mathbb{X}_{\widehat{\mathbf{n}}},\nu_{\widehat{\mathbf{n}}} \right).
	\end{equation*}
	\begin{Rem}
		Here an element $\sigma\in\mathfrak{S}_{N_k}$ acts on $\mathbb{X}_{\widehat{\mathbf{n}}}$ as follows: $\sigma$ maps an element $(x,y)\in\mathbb{X}_{N_k}\times\prod\limits_{j=k+1}^{\infty}\mathbb{X}_{n_j}$ to $(\sigma(x),y)$ (see also \eqref{embedding}).
	\end{Rem}
	
	Define the action of the automorphism $O\in \Aut_0\left( \mathbb{X}_{\widehat{\mathbf{n}}},\nu_{\widehat{\mathbf{n}}} \right)$ on $x=\left( x_1,x_2,\ldots\right)\in \mathbb{X}_{\widehat{\mathbf{n}}}\setminus(n_1-1,n_2-1,\ldots,n_k-1,\ldots) $ in the following way: $Ox=\left( y_1,y_2,\ldots \right)$, where
	\begin{eqnarray*}
		y_p=
		\begin{cases}
			x_p+1\,(\bmod\,n_p),&\text{if}~p\leq\min\left\{i:x_i<n_i-1 \right\},
			\\
			x_p,&\text{if}~p>\min\left\{i:x_i<n_i-1 \right\}.
		\end{cases}
	\end{eqnarray*}
	Also, define $O$ at $(n_1-1,n_2-1,\ldots,n_k-1,\ldots)$ as
	\begin{eqnarray*}
		O(n_1-1,n_2-1,\ldots,n_k-1,\ldots)=(0,0,\ldots,0,\ldots).
	\end{eqnarray*}
	The following fact is immediate.
	\begin{Lm}\label{periodic_approximations}
		Let $x=(x_1,x_2,\ldots)\in\mathbb{X}_{\widehat{\mathbf{n}}}$ and $O^mx=\left(\left( O^mx \right)_1,  \left( O^mx \right)_2, \ldots\right)$. Then
		\begin{itemize}
			\item[\rm a)] the following equalities hold: $\,^j\!x=\,^j\!\!\left( O^{N_j}x\right)$, $O^m\left(\,^j\!\!\mathbb{A}_{\left(^j\!x\right)}\right)=\,^j\!\!\mathbb{A}_{\left( ^{^j}\!\!\left(O^mx\right) \right)}$ and $\bigcup\limits_{m=0}^{N_j-1}\,^j\!\!\mathbb{A}_{\left( ^{^j}\!\!\left(O^mx\right) \right)}=\mathbb{X}_{\widehat{\mathbf{n}}}$;
			
			\item[\rm b)] for the map ${^{^j}_x}\!O$, defined as follows
			\begin{eqnarray*}
				{^{^j}_x}\!Oz=
				\begin{cases}
					Oz, &\text{if}~z\in\bigcup\limits_{m=0}^{N_j-2} O^m\left(\,^j\!\!\mathbb{A}_{\left( ^{^j}\!\!x\right)}\right),
					\\
					O^{-(N_j-1)}z, &\text{if}~z\in  O^{N_j-1}\left(\,^j\!\!\mathbb{A}_{\left( ^{^j}\!\!x\right)}\right),
				\end{cases}
			\end{eqnarray*}
			where $z=(z_1,z_2,\ldots)\in\mathbb{X}_{\widehat{\mathbf{n}}}$, the period of each $z\in\mathbb{X}_{\widehat{\mathbf{n}}}$ equals $N_j$;
			
			\item[\rm c)] for the element $\mathbf{0}=(0,0,\ldots)\in\mathbb{X}_{\widehat{\mathbf{n}}}$ we have ${^{^j}_{\scriptscriptstyle{\mathbf{0}}}}\!Oz=\left( \left({^{^j}_{\scriptscriptstyle{\mathbf{0}}}}\!Oz\right)_1,\left( {^{^j}_{\scriptscriptstyle{\mathbf{0}}}}\!Oz\right)_2,\ldots,\left( {^{^j}_{\scriptscriptstyle{\mathbf{0}}}}\!Oz \right)_j,z_{j+1},z_{j+2},\ldots\right)$, where
			\begin{eqnarray}\label{odometer_periodic}
				\left( {^{^j}_{\scriptscriptstyle{\mathbf{0}}}}\!Oz \right)_p=
				\begin{cases}
					z_p+1\,(\bmod\,n_p), &\text{if}~ p\leq\min\left\{i:z_i<n_i-1 \right\}\leq j,
					\\
					z_p, &\text{if}~ \min\left\{i:z_i<n_i-1 \right\}<p\leq j,
					\\
					0, &\text{if}~p\leq j<\min\left\{i:z_i<n_i-1 \right\}.
				\end{cases}
			\end{eqnarray}
		\end{itemize}
	\end{Lm}
	
	Define an invariant metric $\rho$ on the group $\Aut_0\left( \mathbb{X}_{\widehat{\mathbf{n}}},\nu_{\widehat{\mathbf{n}}} \right)$  as follows
	\begin{eqnarray}\label{metr}
		\rho(\alpha,\beta)=\nu_{\widehat{\mathbf{n}}}\left( x\in \mathbb{X}_{\widehat{\mathbf{n}}}:\alpha(x)\neq\beta(x) \right),\;\alpha, \beta\in \Aut_0\left( \mathbb{X}_{\widehat{\mathbf{n}}},\nu_{\widehat{\mathbf{n}}} \right).
	\end{eqnarray}
	For an automorphism $\alpha\in\Aut_0\left( \mathbb{X}_{\widehat{\mathbf{n}}},\nu_{\widehat{\mathbf{n}}} \right)$ denote by $\left[\alpha\right]$ the subgroup in $\Aut_0\left( \mathbb{X}_{\widehat{\mathbf{n}}},\nu_{\widehat{\mathbf{n}}} \right)$ defined as follows:
	$\beta\in\left[\alpha\right]$, if for almost all $z\in\mathbb{X}_{\widehat{\mathbf{n}}}$ the equality
	\begin{eqnarray}\label{full_GR}
		\beta(z)=\alpha^{d(\beta,z)}(z)
	\end{eqnarray}
	holds, where $d(\beta,\cdot)$ is a measurable function on $\left( \mathbb{X}_{\widehat{\mathbf{n}}},\nu_{\widehat{\mathbf{n}}} \right)$ with values in $\mathbb{Z}$.
Denote by $\Sigma_j$ a $\sigma$-algebra on $\mathbb{X}_{\widehat{\mathbf{n}}}$ generated by collection of the cylindric subsets $\left\{ \,^j\!\!\mathbb{A}_y \right\}$, ${y\in \prod\limits_{k=1}^j\mathbb{X}_{n_k}}$.
	
	\begin{Lm}\label{free_action}
		Let $\overline{\mathfrak{S}}_{\widehat{\mathbf{n}}}$ be the closure of the group $\mathfrak{S}_{\widehat{\mathbf{n}}}$ with respect to the metric $\rho$. Then
		\begin{itemize}
			\item[\rm a)] $\mathfrak{S}_{N_j}=\left\{\beta\in\left[  {^{^j}_{\scriptscriptstyle{\mathbf{0}}}}\!O\right]:\beta\Sigma_j=\Sigma_j\right\}$;
			
			\item[\rm b)] $O\in \overline{\mathfrak{S}}_{\widehat{\mathbf{n}}}$;
			
			\item[\rm c)] the action of automorphism $O$ on $\left(\mathbb{X}_{\widehat{\mathbf{n}}},\nu_{\widehat{\mathbf{n}}}\right)$ is ergodic;
			
			\item[\rm d)] for every $l\in \mathbb{Z}\setminus 0$ the automorphism $O^l$ acts freely on $\mathbb{X}_{\widehat{\mathbf{n}}}$; i.e. if there is an $x\in\mathbb{X}_{\widehat{\mathbf{n}}}$ such that $O^lx=x$, then $l=0$.
		\end{itemize}
	\end{Lm}
	
	\begin{proof}
		The property {\rm a)} is a consequence of Lemma \ref{periodic_approximations} {\rm c)}. From the parts {\rm b)} and {\rm c)} of Lemma \ref{periodic_approximations} we have
		\begin{eqnarray*}
			\nu_{\widehat{\mathbf{n}}}\left(x\in\mathbb{X}_{\widehat{\mathbf{n}}}:Ox\neq  {^{^j}_{\scriptscriptstyle{\mathbf{0}}}}\!Ox \right)\leq\frac{1}{N_j}.
		\end{eqnarray*}
		Taking this and part {\rm a)} into account, we obtain the part {\rm b)} of Lemma \ref{free_action}. Therefore, the ergodicity of the automorphism $O$ is equivalent to the ergodicity of the action of $\overline{\mathfrak{S}}_{\widehat{\mathbf{n}}}$. And finally, the property {\rm d)}  follows from the definition of the automorphism $O$.
	\end{proof}
	
	Lemma \ref{periodic_approximations} ({\rm b},{\rm c}) and Lemma \ref{free_action} ({\rm a}) imply that action of each automorphism $\sigma\in\mathfrak{S}_{\widehat{\mathbf{n}}}$ on $x\in \mathbb{X}_{\widehat{\mathbf{n}}}$ can be expressed in the following way
	\begin{eqnarray}\label{degree_unique}
		\sigma(x)=O^{d(\sigma,x)}(x),
	\end{eqnarray}
	where $d(\sigma,x)\in\mathbb{Z}$. The uniqueness of the function $d(\sigma,\cdot)$ in \eqref{degree_unique} is a consequence of Lemma \ref{free_action} ({\rm d}). Note that if  $\gamma,\sigma\in\mathfrak{S}_{\widehat{\mathbf{n}}}$, then
	\begin{eqnarray}\label{cocycle}
		d(\gamma\sigma,x)=d(\gamma,\sigma(x))+d(\sigma,x).
	\end{eqnarray}
	
	\begin{Rem}
		It follows from (\ref{full_GR}) and (\ref{degree_unique}) that $[O]=\overline{\mathfrak{S}_{\widehat{\mathbf{n}}}}$.
	\end{Rem}
	
	\subsection{Construction of a ${\rm II}_1$ factor representation of the group $\overline{\mathfrak{S}}_{\widehat{\mathbf{n}}}$}\label{factor_repr_constr}
	
	In the Hilbert space $\mathcal{H}=L^2\left( \mathbb{X}_{\widehat{\mathbf{n}}},\nu_{\widehat{\mathbf{n}}} \right)\otimes l^2(\mathbb{Z})$ define the unitary operator $\mathcal{F}(\sigma)$, where $\sigma\in\mathfrak{S}_{\widehat{\mathbf{n}}}$, as follows:
	\begin{eqnarray}\label{fundamental_repr}
		\left( \mathcal{F}(\sigma)\eta\right)(x,m)=\eta\left(\sigma^{-1}(x),m-d(\sigma^{-1},x)\right) \text{ for all } \eta\in \mathcal{H}.
	\end{eqnarray}
	Equality  \eqref{cocycle} implies that the map $\sigma\ni\mathfrak{S}_{\widehat{\mathbf{n}}}\mapsto \mathcal{F}(\sigma)$ is a unitary representation of the group $\mathfrak{S}_{\widehat{\mathbf{n}}}$, which can be extended by continuity with respect to the metric $\rho$ (see \eqref{metr}) to the representation of the group $\overline{\mathfrak{S}}_{\widehat{\mathbf{n}}}$. Thus   \eqref{fundamental_repr} define the representation of the group  $\overline{\mathfrak{S}}_{\widehat{\mathbf{n}}}$.
	
	Denote by $\mathcal{B}(\mathcal{H})$ the set of all bounded linear operators acting on $\mathcal{H}$.
	Put
	$$
	\mathcal{F}\left( \mathfrak{S}_{\widehat{\mathbf{n}}}\right)^\prime=\left\{A\in \mathcal{B}(\mathcal{H}): AB=BA \text{ for all }\; B\in \mathcal{F}\left( \mathfrak{S}_{\widehat{\mathbf{n}}}\right)\right\}.
	$$
	Denote by $\mathcal{F}\left( \mathfrak{S}_{\widehat{\mathbf{n}}}\right)^{\prime\prime}$ the $w^*$-algebra generated by operators $\mathcal{F}\left( \mathfrak{S}_{\widehat{\mathbf{n}}}\right)$.
	Let $\gimel$ be the function on $\mathbb{X}_{\widehat{\mathbf{n}}}$ that is identically one on $\mathbb{X}_{\widehat{\mathbf{n}}}$. Define function $\delta_i$ on $\mathbb{Z}$, where $i\in\mathbb{Z}$ as
	$\delta_i(m)=\left\{
	\begin{array}{rl}
		1,\text{ if } m=i\\
		0,\text{ if } m\neq i.
	\end{array}
	\right.$
	Put $\xi_0=\gimel\otimes\delta_0$. It is easy to check that
	\begin{eqnarray*}
		\left(\mathcal{F}(s)\mathcal{F}(\sigma)\xi_0,\xi_0\right)=\left(\mathcal{F}(\sigma)\mathcal{F}(s)\xi_0,\xi_0\right) \text{ for all } \sigma,s \in \mathfrak{S}_{\widehat{\mathbf{n}}}.
	\end{eqnarray*}
	Thus, the vector state $\tr$ on $\mathcal{F}\left(\mathfrak{S}_{\widehat{\mathbf{n}}}\right)^{\prime\prime}$ defined as
	\begin{eqnarray}\label{trace}
		\tr(A)=\left( A\xi_0,\xi_0 \right),\;A\in \mathcal{F}\left(\mathfrak{S}_{\widehat{\mathbf{n}}}\right)^{\prime\prime},
	\end{eqnarray}
	is  central. Namely, the equality $\tr (AB)=\tr (BA)$ holds for all $A,B\in\mathcal{F}\left(\mathfrak{S}_{\widehat{\mathbf{n}}}\right)^{\prime\prime}$.
	In particular, it follows from (\ref{chi_nat}) that for any $s\in\mathfrak{S}_{\widehat{\mathbf{n}}}$ we have	
\begin{eqnarray}\label{natural_character}
		\chi_{\nat}(s)=\tr \left(\mathcal{F}(s)\right).
	\end{eqnarray}
	Indeed, Lemma \ref{free_action} ({\rm d}) and formulas \eqref{fundamental_repr} and \eqref{chi_nat} imply that for $\sigma\in\mathfrak{S}_{N_k}\subset\mathfrak{S}_{\widehat{\mathbf{n}}}$ we have
	\begin{eqnarray}\label{character_fundamental}
		\tr \left(\mathcal{F}(\sigma)\right)=\nu_{\widehat{\mathbf{n}}}\left( x\in \mathbb{X}_{\widehat{\mathbf{n}}}:\sigma x=x\right)=\frac{\#\left\{ x\in\mathbb{X}_{N_k}:\sigma x=x\right\}}{N_k}=\chi_{\nat}(\sigma).
	\end{eqnarray}

	Now consider two families of operators $\left\{\mathfrak{M}^\prime(f) \right\}_{f\in L^\infty\left( \mathbb{X}_{\widehat{\mathbf{n}}},\nu_{\widehat{\mathbf{n}}} \right)}$ and $\left\{\mathcal{F}^\prime(O^k) \right\}_{k\in\mathbb{Z}}$ which belong to $\mathcal{F}\left( \mathfrak{S}_{\widehat{\mathbf{n}}}\right)^\prime$ and whose action on an element $\eta\in\mathcal{H}$  is defined in the following way:
	\begin{equation}\label{right_component}
		\left(\mathfrak{M}^\prime(f)\eta\right)(x,m)=f\left( O^mx\right)\eta(x,m),\;\left( \mathcal{F}'(O^k)\eta \right)(x,m)=\eta(x,m-k).
	\end{equation}
	Using \eqref{fundamental_repr}, one can check that $\mathfrak{M}^\prime(f)$ and $\mathcal{F}'(O^k)$ belong to $\mathcal{F}\left( \mathfrak{S}_{\widehat{\mathbf{n}}}\right)^\prime$. Thus, we obtain the following statement:
	\begin{Lm}\label{cyclic_lemma}
    		Let $\mathcal{N}^\prime$ be the $w^*$-subalgebra of $\mathcal{F}\left( \mathfrak{S}_{\widehat{\mathbf{n}}}\right)^\prime$, which is generated by the operators $\left\{\mathfrak{M}^\prime(f)\right\}_{f\in L^\infty\left( \mathbb{X}_{\widehat{\mathbf{n}}},\nu_{\widehat{\mathbf{n}}} \right)}$ and $\mathcal{F}'(O)$. Then, vector $\xi_0$ is cyclic for $\mathcal{N}^\prime$, i.e. the closure of the set $\mathcal{N}^\prime\xi_0$ coincides with $\mathcal{H}$.
	\end{Lm}
	
	\begin{Lm}\label{mult_operator}
		Let $y\in \prod\limits_{k=1}^j\mathbb{X}_{n_k}$,  $\mathfrak{I}_{\,^j\!\!\mathbb{A}_y}(x)=\left\{
		\begin{array}{rl}
			1, \text{ if } x\in\,^j\!\!\mathbb{A}_y\\
			0,\text{ if } x\notin\,^j\!\!\mathbb{A}_y
		\end{array}
		\right.$ (see \eqref{cylinder}) and operator $\mathfrak{M}\left(\mathfrak{I}_{\,^j\!\!\mathbb{A}_y}\right)$ acts on $\eta\in\mathcal{H}$ as follows
		\begin{eqnarray}
			\left(\mathfrak{M}\left(\mathfrak{I}_{\,^j\!\!\mathbb{A}_y}\right)\eta\right)(x,m)=\mathfrak{I}_{\,^j\!\!\mathbb{A}_y}(x)\eta(x,m).
		\end{eqnarray}
		Then, operator $\mathfrak{M}\left(\mathfrak{I}_{\,^j\!\!\mathbb{A}_y}\right)$ belongs to algebra $\mathcal{F}\left( \mathfrak{S}_{\widehat{\mathbf{n}}}\right)^{\prime\prime}$.
	\end{Lm}
	\begin{proof}
		Applying lemma \ref{periodic_approximations} ({\rm a}) we obtain that
		\begin{eqnarray}
			O^{N_k}\,^j\!\!\mathbb{A}_y=\,^j\!\!\mathbb{A}_y \text{ for all } k\geq j.
		\end{eqnarray}
		For $k\geq j$ define an automorphism $\,^k\!D_y$ as
		\begin{eqnarray}\label{def_D}
			\,^k\!D_yx=\left\{
			\begin{array}{rl}
				x, \text{ if } x\in\,^j\!\!\mathbb{A}_y\\
				O^{N_k}x,\text{ if } x\notin\,^j\!\!\mathbb{A}_y.
			\end{array}
			\right.
		\end{eqnarray}
		In view of Lemma \ref{free_action} ({\rm b}) the automorphism $\,^k\!D_y$ belongs to the group $\overline{\mathfrak{S}}_{\widehat{\mathbf{n}}}$. Since the representation $\mathcal{F}$ of $\mathfrak{S}_{\widehat{\mathbf{n}}}$ can be extended to a representation of $\overline{\mathfrak{S}}_{\widehat{\mathbf{n}}}$ (see \eqref{fundamental_repr}) it suffices to prove that
		\begin{eqnarray}\label{weak_lim}
			w-\lim\limits_{k\to\infty} \mathcal{F}\left(\,^k\!D_y\right)=\mathfrak{M}\left(\mathfrak{I}_{\,^j\!\!\mathbb{A}_y}\right).
		\end{eqnarray}
		Here ``$w-\lim$''\ stands for the limit in the weak operator topology.
		
		Using \eqref{fundamental_repr} and \eqref{def_D} we obtain
		\begin{eqnarray*}
			\mathcal{F}\left(\,^k\!D_y\right)\xi_0=\mathfrak{I}_{\,^j\!\!\mathbb{A}_y}\otimes\delta_0+\left(\gimel- \mathfrak{I}_{\,^j\!\!\mathbb{A}_y}\right)\otimes \delta_{N_k}.
		\end{eqnarray*}
From here, taking into account the weak convergence of the sequence  $\delta_{N_k}$ to zero-vector in $l^2(\mathbb{Z})$, we get that the sequence  $\mathcal{F}\left(\,^k\!D_y\right)\xi_0$ converges weakly  in $\mathcal{H}$ to the vector  $\mathfrak{I}_{\,^j\!\!\mathbb{A}_y}\otimes\delta_0=\mathfrak{M}\left(\mathfrak{I}_{\,^j\!\!\mathbb{A}_y}\right)\xi_0$

		Now \eqref{weak_lim} is a consequence of Lemma \ref{cyclic_lemma}.
	\end{proof}
	
	The following statement is a direct corollary of Lemma \ref{mult_operator}.
	\begin{Co}\label{Coll}
		The algebra $\mathcal{F}\left( \mathfrak{S}_{\widehat{\mathbf{n}}}\right)^{\prime\prime}$ contains the family of operators $\left\{\mathfrak{M}\left( f \right) \right\}_{f\in L^\infty\left( \mathbb{X}_{\widehat{\mathbf{n}}},\nu_{\widehat{\mathbf{n}}} \right)}$, acting on an element $\eta\in\mathcal{H}$ as follows
		\begin{eqnarray}\label{mult_oper}
			\left(\mathfrak{M}\left( f \right)\eta\right)(x,m)= f(x)\eta(x,m).
		\end{eqnarray}
		In particular, vector $\xi_0$ is a cyclic vector for the algebra $\mathcal{F}\left( \mathfrak{S}_{\widehat{\mathbf{n}}}\right)^{\prime\prime}$ (see Lemma \ref{cyclic_lemma}) and due to \eqref{fundamental_repr} the following relations hold
		\begin{eqnarray}\label{covariance}
			\mathcal{F}(\sigma)\mathfrak{M}(f)\left( \mathcal{F}(\sigma)\right)^{-1}=\mathfrak{M}(\,^\sigma \!\!f), \text{ where } \,^\sigma \!\!f(x)=f\left( \sigma^{-1}(x) \right), \;\sigma\in \overline{\mathfrak{S}}_{\widehat{\mathbf{n}}}.
		\end{eqnarray}
	\end{Co}
	\begin{Prop}\label{factor}
		The algebra $\mathcal{F}\left( \mathfrak{S}_{\widehat{\mathbf{n}}}\right)^{\prime\prime}$ is a ${\rm II}_1$-factor.
	\end{Prop}
	\begin{proof}
		Suppose that operator $A$ belongs to $\mathcal{F}\left( \mathfrak{S}_{\widehat{\mathbf{n}}}\right)^{\prime\prime}\cap\mathcal{F}\left( \mathfrak{S}_{\widehat{\mathbf{n}}}\right)^{\prime}$.
		Since $\mathcal{H}=L^2\left( \mathbb{X}_{\widehat{\mathbf{n}}},\nu_{\widehat{\mathbf{n}}} \right)\otimes l^2(\mathbb{Z})$, we have
		\begin{eqnarray*}
			A\xi_0=\sum\limits_{i\in\mathbb{Z}}f_i\otimes\delta_i \;\text{ where } f_i\in L^2\left( \mathbb{X}_{\widehat{\mathbf{n}}},\nu_{\widehat{\mathbf{n}}} \right),\;\delta_i(m)=\left\{
			\begin{array}{rl}
				1,& \text{ if } m=i
				\\
				0,& \text{ if } m\neq i.
			\end{array}
			\right.
		\end{eqnarray*}
		Recall that $O\in \overline{\mathfrak{S}}_{\widehat{\mathbf{n}}}$ and representation $\mathcal{F}$ of $\mathfrak{S}_{\widehat{\mathbf{n}}}$ can be extended by continuity with respect to the metric $\rho$ to a representation of $\overline{\mathfrak{S}}_{\widehat{\mathbf{n}}}$. Therefore, the following equality holds
		\begin{eqnarray}\label{action_oper_from_center}
			A\xi_0=\sum\limits_{i\in\mathbb{Z}}\mathcal{F}\left( O^{-i}\right)\left(\,^{^{(O^{i})}}\!\!f_i \otimes \delta_0\right).
		\end{eqnarray}
		Hence, for any $f\in L^\infty\left( \mathbb{X}_{\widehat{\mathbf{n}}},\nu_{\widehat{\mathbf{n}}} \right)$ we have
		\begin{eqnarray}
			A\mathfrak{M}(f)\xi_0\stackrel{\eqref{mult_oper}}=\mathfrak{M}(f)\,A\xi_0=
			\sum\limits_{i\in\mathbb{Z}}ff_i\otimes\delta_i.
		\end{eqnarray}
		We thus get
		\begin{eqnarray*}
			\|A\|^2\int\limits_{\mathbb{X}_{\widehat{\mathbf{n}}}}|f(x)|^2\,{\rm d}\,\nu_{\widehat{\mathbf{n}}}=\|A\|^2\;\|\mathfrak{M}(f)\xi_0\|^2\geq \int\limits_{\mathbb{X}_{\widehat{\mathbf{n}}}}|f(x)|^2\left(\sum\limits_{i\in\mathbb{Z}} |f_i(x)|^2 \right)\,{\rm d}\,\nu_{\widehat{\mathbf{n}}}.
		\end{eqnarray*}
		It follows from this that
		\begin{eqnarray*}
			\left\|  \sum\limits_{i\in\mathbb{Z}} |f_i|^2\right\|_{ L^\infty\left( \mathbb{X}_{\widehat{\mathbf{n}}},\nu_{\widehat{\mathbf{n}}} \right)}\leq \|A\|^2.
		\end{eqnarray*}
		In particular, $f_i\in  L^\infty\left( \mathbb{X}_{\widehat{\mathbf{n}}},\nu_{\widehat{\mathbf{n}}} \right)$ and $\mathfrak{M}(f_i)\in \mathcal{F}\left( \mathfrak{S}_{\widehat{\mathbf{n}}}\right)^{\prime\prime}$ for all $i$ (see Corollary \ref{Coll}).
	Hence, 	using \eqref{action_oper_from_center}, we obtain
		\begin{eqnarray}
			A\xi_0=\sum\limits_{i\in\mathbb{Z}}\mathcal{F}\left( O^{-i}\right)\mathfrak{M}\left(\,^{^{(O^{i})}}\!\!f_i\right)\xi_0\stackrel{\eqref{covariance}}=\sum\limits_{i\in\mathbb{Z}}
			\mathfrak{M}\left( f_i \right)\mathcal{F}\left( O^{-i}\right)\xi_0.
		\end{eqnarray}
		The last equality, Lemma \ref{cyclic_lemma} and Corollary \ref{Coll} imply that
		\begin{eqnarray}
			A=\sum\limits_{i\in\mathbb{Z}}
			\mathfrak{M}\left( f_i \right)\mathcal{F}\left( O^{-i}\right).
		\end{eqnarray}
		Therefore, the equality $\mathcal{F}(O)A\xi_0=A\mathcal{F}(O)\xi_0$ is equivalent to relations $\,^{^O}\!\!f_i=f_i$, where $i\in\mathbb{Z}$. It follows from Lemma \ref{free_action} that function $f_i$ should be constant almost everywhere. In other words, there are constants $c_i$, $i\in\mathbb{Z}$ such that $f_i\equiv c_i$ almost everywhere and $A\xi_0=\sum\limits_{i\in\mathbb{Z}}c_i\;\mathcal{F}(O^{-i})\xi_0$.
		
		Finally, note that equality $\mathfrak{M}(f)\; A\xi_0=A\;\mathfrak{M}(f)\xi_0$, $f\in L^\infty\left( \mathbb{X}_{\widehat{\mathbf{n}}},\nu_{\widehat{\mathbf{n}}} \right)$ is equivalent to
		\begin{eqnarray*}
			c_i\cdot\;f=c_i \cdot\,^{^{(O^{-i})}}\!\!f,\;\;i\in\mathbb{Z}.
		\end{eqnarray*}
		Since $f$  is arbitrary, Lemma \ref{free_action} implies that $c_i=0$ for all $i\neq 0$. Therefore, $A\xi_0=c_0\xi_0$.  By lemma \ref{cyclic_lemma},  $A$ is a scalar operator.
	\end{proof}
	
	\subsection{Construction of an irreducible representation of the group $\overline{\mathfrak{S}}_{\widehat{\mathbf{n}}}\times \overline{\mathfrak{S}}_{\widehat{\mathbf{n}}}$}
	
	For the ${\rm II}_1$ factor representation $\mathcal{F}$ there is a corresponding irreducible representation $\mathcal{F}^{(2)}$ of the group $\overline{\mathfrak{S}}_{\widehat{\mathbf{n}}}\times \overline{\mathfrak{S}}_{\widehat{\mathbf{n}}}$ acting in the Hilbert space $\mathcal{H}=L^2\left( \mathbb{X}_{\widehat{\mathbf{n}}},\nu_{\widehat{\mathbf{n}}} \right)\otimes l^2(\mathbb{Z})$ such that
	\begin{eqnarray}\label{square}
		\begin{aligned}
			&\mathcal{F}^{(2)}\left(g,\id\right)=\mathcal{F}(g) \text{ where } \id \text{ is the identity of } \overline{\mathfrak{S}}_{\widehat{\mathbf{n}}},\; g\in\overline{\mathfrak{S}}_{\widehat{\mathbf{n}}};
			\\
			&\mathcal{F}^{(2)}(g,g)\xi_0=\xi_0 \text{ for all } g\in\overline{\mathfrak{S}}_{\widehat{\mathbf{n}}} \text{ and } \mathcal{F}^{(2)}\left(\id,\overline{\mathfrak{S}}_{\widehat{\mathbf{n}}}\right)\subset \mathcal{F}\left( \overline{\mathfrak{S}}_{\widehat{\mathbf{n}}}\right)^{\prime};
			\\
			&\mathcal{F}^{(2)}\left(\id,O^k\right)=\mathcal{F}'(O^k) \text{ (see \eqref{right_component})}.
		\end{aligned}
	\end{eqnarray}
	In order to define $\mathcal{F}^{(2)}$ let us introduce the antiunitary operator $\mathcal{J}$ acting on $\mathcal{H}$ as follows:
	\begin{eqnarray}\label{anti_unitary}
		\left(\mathcal{J}\eta\right)(x,m)=\overline{\eta\left( O^mx,-m \right)}, \;\eta\in \mathcal{H}=L^2\left( \mathbb{X}_{\widehat{\mathbf{n}}},\nu_{\widehat{\mathbf{n}}} \right)\otimes l^2(\mathbb{Z}).
	\end{eqnarray}
	Then direct calculations show that
	\begin{eqnarray*}
		\mathcal{J}\,\mathcal{F}(g)\xi_0=\mathcal{F}\left(g^{-1}\right)\xi_0=\left(\mathcal{F}(g)\right)^*\xi_0,\; g\in \mathfrak{S}_{\widehat{\mathbf{n}}};\\
		\left(\mathcal{J}\eta,\mathcal{J}\zeta\right)=\left(\zeta,\eta\right)\text{ for all } \zeta,\eta\in\mathcal{H}.
	\end{eqnarray*}
	Combining \eqref{fundamental_repr}, \eqref{right_component}, \eqref{mult_oper} and \eqref{anti_unitary} we obtain
	\begin{eqnarray}
		\mathcal{J}\mathfrak{M}(f)\mathcal{J}=\mathfrak{M}'(\overline{f}),\; \mathcal{J}\mathcal{F}(O^k)\mathcal{J}=\mathcal{F}'(O^k).
	\end{eqnarray}
	These equalities, Lemma \ref{cyclic_lemma} and Corollary \ref{Coll} imply that
	\begin{eqnarray}\label{include}
		\mathcal{N}'=\mathcal{J}\mathcal{F}\left( \mathfrak{S}_{\widehat{\mathbf{n}}}\right)^{''}\mathcal{J}\subset\mathcal{F}\left( \mathfrak{S}_{\widehat{\mathbf{n}}}\right)^\prime.
	\end{eqnarray}
	
	Therefore, the operators $\left\{\mathcal{F}^{(2)}(g,h)=
	\mathcal{F}(g)\mathcal{J}\mathcal{F}(h)\mathcal{J}\right\}_{g,h\in\mathfrak{S}_{\widehat{\mathbf{n}}}}$ define a representation of the group $\mathfrak{S}_{\widehat{\mathbf{n}}}\times\mathfrak{S}_{\widehat{\mathbf{n}}}$. It is easy to check  that $\mathcal{F}^{(2)}$ satisfies conditions \eqref{square}.
	\begin{Prop}
		The representation $\mathcal{F}^{(2)}$ is irreducible.
	\end{Prop}
	\begin{proof}
		In view of Proposition \ref{factor} it suffices to show that
		\begin{eqnarray}\label{equality}
			\mathcal{J}\mathcal{F}\left( \mathfrak{S}_{\widehat{\mathbf{n}}}\right)^{''}\mathcal{J}=\mathcal{F}\left( \mathfrak{S}_{\widehat{\mathbf{n}}}\right)^\prime.
		\end{eqnarray}
		Denote for convenience the factor $\mathcal{F}\left( \mathfrak{S}_{\widehat{\mathbf{n}}}\right)^{''}$ as $M$. Consider arbitrary bounded operators $A'$ and $B$ such that $A'\in M'$ and $B\in \mathcal{J}M'\mathcal{J}$.
		In order to prove \eqref{equality} it is enough to check that
		\begin{eqnarray}\label{A'_B_permutation}
			A'B=BA'.
		\end{eqnarray}
		According to Corollary \ref{Coll} vector $\xi_0$ is a cyclic vector for $M$. Hence, there is a sequence $\left\{B_n \right\}_{n=1}^\infty\subset M$ such that
		\begin{eqnarray}\label{B_approximation}
			\left\|B\xi_0-B_n\xi_0  \right\|_\mathcal{H}=0.
		\end{eqnarray}
		In particular,
		\begin{eqnarray}\label{sequence_fundamental}
			\lim\limits_{l,m\to\infty}\left\|B_l\xi_0-B_m\xi_0  \right\|_\mathcal{H}=0.
		\end{eqnarray}
		Since $\left( UV\xi_0,\xi_0 \right)=\left( VU\xi_0,\xi_0 \right)$ for all $U,V\in M$ we have
		\begin{eqnarray}\label{norm_central_equality}
			\left\|B_l^*\xi_0-B_m^*\xi_0  \right\|_\mathcal{H}=\left\|B_l\xi_0-B_m\xi_0  \right\|_\mathcal{H}.
		\end{eqnarray}
		Relations \eqref{sequence_fundamental}, \eqref{norm_central_equality} and equalities $V'B_n=B_nV'$ imply that the sequence $B_n^*V'\xi_0$ is a Cauchy sequence for any $V'\in M'$. Namely,
		\begin{eqnarray*}
			\lim_{l,m\to\infty}\left\|B_l^*V'\xi_0-B_m^*V'\xi_0  \right\|_\mathcal{H}=\lim_{l,m\to\infty}\|V'(B_l^*\xi_0-B_m^*\xi_0)\|=0.
		\end{eqnarray*}
		Therefore, we can define the linear operators $\widehat{B}$ and  $B^\sharp$ as follows
		\begin{eqnarray*}
			\widehat{B}V'\xi_0=\lim_{n\to\infty}B_nV'\xi_0,\;B^\sharp V'\xi_0=\lim_{n\to\infty}B_n^*V'\xi_0,\;\;V'\in M'.
		\end{eqnarray*}
Denote by $\mathcal{D}\left(\widehat{B}\right)$ and $\mathcal{D}\left(B^\sharp\right)$ the domains of the operators $\mathcal{D}\left(\widehat{B} \right)$ and $\mathcal{D}\left( B^\sharp \right)$. It is clear that $M'\xi_0\subset\mathcal{D}\left(\widehat{B} \right)$
and 	$M'\xi_0\subset\mathcal{D}\left( B^\sharp \right)$.	For any $W'\in \mathcal{J}M\mathcal{J}\subset M'$ and $V'\in M'$, we have $W'B=BW'$. Now, it follows from  \eqref{B_approximation} that
		\begin{eqnarray*}
			\begin{aligned}
				\left( BW'\xi_0,V'\xi_0\right)
				&=\left( W'B\xi_0,V'\xi_0\right)
				\stackrel{\eqref{B_approximation}}{=}\lim_{n\to\infty}\left( W'B_n\xi_0,V'\xi_0\right)=
				\\
				&=\lim_{n\to\infty}\left( B_nW'\xi_0,V'\xi_0\right)=\lim_{n\to\infty}\left(W'\xi_0,B_n^*V'\xi_0\right)=
				\\
				&=\left( W'\xi_0,B^\sharp V'\xi_0\right).
			\end{aligned}
		\end{eqnarray*}
		Thus, $\left(W'\xi_0,B^*V'\xi_0  \right)=\left( W'\xi_0,B^\sharp V'\xi_0\right)$ for all $W'\in\mathcal{J}M\mathcal{J}=\mathcal{N}'$ and $V'\in M'$ (see Lemma \ref{cyclic_lemma}). By Lemma \ref{cyclic_lemma} the set $\left\{W'\xi_0 \right\}_{W'\in\mathcal{J}M\mathcal{J}}$ is norm dense in $\mathcal{H}$. Therefore,
		\begin{eqnarray}\label{equality_conju}
			\left(\eta,B^*V'\xi_0  \right)=\left(\eta,B^\sharp V'\xi_0\right)
		\end{eqnarray}
		for all $\eta\in\mathcal{H}$ and $V'\in M'$.
		Hence, for arbitrary $U',V'\in\mathcal{J}M\mathcal{J}=\mathcal{N}'$ we have
		\begin{eqnarray*}
			\begin{aligned}
				\left( A'BU'\xi_0,V'\xi_0\right)
				&=\left( A'U'B\xi_0,V'\xi_0\right)=\lim_{n\to\infty}\left( A'U'B_n\xi_0,V'\xi_0\right)=
				\\
				&=\lim_{n\to\infty}\left( B_nA'U'\xi_0,V'\xi_0\right)=\lim_{n\to\infty}\left( A'U'\xi_0,B_n^*V'\xi_0\right)
				\\
				&=\left( A'U'\xi_0,B^\sharp V'\xi_0\right)\stackrel{\eqref{equality_conju}}{=}
				=\left(A'U'\xi_0,B^* V'\xi_0\right)=\left(BA'U'\xi_0, V'\xi_0\right).
			\end{aligned}
		\end{eqnarray*}
		This proves the equality \eqref{A'_B_permutation} and concludes the proof.
	\end{proof}
	
	\section{Preliminaries about the representation theory of the symmetric groups}
	
	In this section we remind some notions and facts from the representation theory of the symmetric groups $\mathfrak{S}_{N}$ which are used in the proof of the main theorem.
	
	\subsection{The minimal element of the conjugacy class}\label{minimal_element}
	
	Consider the symmetric group $\mathfrak{S}_{N}$. Denote by  $\mathfrak{C}_{\mathfrak{m}}$ the conjugacy class of the group $\mathfrak{S}_{N}$, which consists of permutations of the cycle type $\mathfrak{m}=(m_1,m_2,\ldots,m_N)$, where $m_i$ is the number of independent cycles of length $i$.
	\begin{Def}
		The \emph{support} of a permutation $s\in\mathfrak{S}_{N}$ is the subset $\supp_{_{N}}\,s=\left\{x\in \mathbb{X}_{N}\mid sx\neq x \right\}$.
	\end{Def}
	\begin{Rem}\label{convention_cycle_type}
		Further, we also use the notation $\mathfrak{m}=(m_1,m_2,\ldots,m_l)$ for a cycle type $\mathfrak{m}$ instead of $\mathfrak{m}=(m_1,m_2,\ldots,m_N)$ if $m_{l+1}=\ldots=m_{N}=0$ (here $l<N$).
	\end{Rem}
	Clearly, for $\sigma\in\mathfrak{C}_{\mathfrak{m}}\subset\mathfrak{S}_{N}$ we have $\sum\limits_{i=1}^N i\cdot m_i=N$ and $\#\supp_{_{N}}\,s \,=\sum\limits_{i=2}^N i\,\cdot m_i$.
	For any distinct elements $n_0,n_1,\ldots,n_{j-1}\in\mathbb{X}_{N}$ denote by $(n_0\;n_1\;\ldots\;n_{j-1})$ the cyclic permutation $c\in\mathfrak{S}_{N}$ such that $c(n_i)=n_{i+1({\rm mod}\,j)}$. Denote by $s_i$ the transposition $(i\;i+1)$. The elements $\{s_1,s_2,\ldots,s_{N-1}\}$ are also known as the \emph{Coxeter generators} of the symmetric group $\mathfrak{S}_{N}$.
	
	\begin{Def}
		Consider a conjugacy class $\mathfrak{C}_{\mathfrak{m}}$ of the symmetric group $\mathfrak{S}_{N}$. Let $\{i_1,i_2,\ldots,i_p\}$, where $1<i_1<i_2<\ldots<i_p$, be the set of those $i\in\{2,3,\ldots,N\}$ for which $m_{i}\geq 1$. The {\it minimal element} of the conjugacy class $\mathfrak{C}_{\mathfrak{m}}$ is the permutation $\sigma_{\mathfrak{m}}$ defined as follows:
		\begin{eqnarray*}\label{minimal_element_def}
			\begin{aligned}
				\sigma\!_{\mathfrak{m}}=
				&\left(1\;2\;\ldots\;i_1 \right)\cdots\left( \left(m_{i_1}-1\right)i_1+1\;\;  \left(m_{i_1}-1\right)i_1+2\;\;\ldots\;\;  m_{i_1}\,i_1\right)\cdots
				\\
				&\left(m_{i_1}\,i_1+1\;\;\; m_{i_1}\,i_1+2\;\;\;\ldots\;\;\;m_{i_1}\,i_1+i_2\right)\cdots
				\\
				&\left(m_{i_1}\,i_1+ \left(m_{i_2}-1\right)i_2+1\;\;\;  m_{i_1}\,i_1+\left(m_{i_2}-1\right)i_2+2\;\;\;\ldots\;\; \;m_{i_1}\,i_1+m_{i_2}\,i_2\right)\cdots
			\end{aligned}
		\end{eqnarray*}
	\end{Def}
	It is clear that $\sum\limits_{q=1}^{p}i_q m_{i_q}=N-m_1$ and $\supp_{_{N}} \,\sigma\!_{\mathfrak{m}}=\left\{1,2,\ldots, N-m_1\right\}\subset\mathbb{X}_{N}$. The crucial property of $\sigma_{\mathfrak{m}}$ is the following decomposition into the product of Coxeter generators:
	\begin{eqnarray}\label{minimal_el_property}
		\sigma\!_{\mathfrak{m}}=s_{j_1}s_{j_2}\cdots s_{j_r},
	\end{eqnarray}
	where $j_1<j_2<\ldots<j_r$ are elements of $\{1,2,\ldots,N\}$ and $r=N-m_1-\sum\limits_{q=1}^p m_{i_q}$. The existence of such decomposition follows from the equality
	\begin{equation*}
		(i\;i+1\;\ldots\;i+j)=s_{i}s_{i+1}\ldots s_{i+j-1}.
	\end{equation*}
	
	For any permutation $\sigma\in\mathfrak{S}_{N}$ denote by $\kappa_{_{N}}(\sigma)$ the minimal number of factors in the decomposition of $\sigma$ into the product of transpositions. Define the \emph{sign (or signature)} of permutation $\sigma$ as $\sgn_{N}(\sigma)=(-1)^{\kappa_{_{N}}(\sigma)}$.
	
	\subsection{Young tableux and diagrams}
	Let $\lambda=(\lambda_1,\lambda_2,\ldots,\lambda_j)$, where $\lambda_1\ge\ldots\ge\lambda_j\ge1$, be a partition of a positive integer $N=|\lambda|$ into positive integer summands  $\lambda_1,\lambda_2,\ldots,\lambda_j$ ($\lambda\vdash N$); i.e. $\sum\limits_{i=1}^j\lambda_i=N$. Denote by $\lambda$ the corresponding Young diagram consisting of $j$ rows of the length $\lambda_1,\lambda_2,\ldots,\lambda_j$ (the $i$-th row consists of $\lambda_i$ boxes).
	The conjugate or transposed Young diagram $\lambda'$ is obtained from $\lambda$ by replacement of its rows by its columns.
	
	A standard Young tableau $T$ of the shape $\lambda$ is the diagram $\lambda$, whose boxes are filled with positive integers from $1$ to $N$ such that each number occurs exactly once and $n_{pq}>\min\,\left\{n_{p+1\,q},n_{p\,q+1} \right\}$ for all $p,q$ (see Figure \ref{figure}). In this paper we consider only standard Young tableaux.
	
	\begin{figure}[!h]
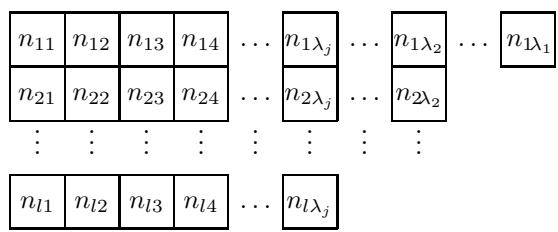

		\centering
		\ytableausetup{boxsize=2em}
		\begin{ytableau}
			n_{11}& n_{12} & n_{13} & n_{14}& \none[\dots]&n_{1\lambda_j}&\none[\dots]&n_{1\lambda_2} &\none[\dots] & n_{1\!\lambda_1}\!\\
			n_{21}    & n_{22}   & n_{23}   & n_{24}& \none[\dots] &n_{2\lambda_j}&\none[\dots]&n_{2\!\lambda_2}  & \none \\
			\none[\vdots]&\none[\vdots]&\none[\vdots] & \none[\vdots] & \none[\vdots]&\none[\vdots] &\none[\vdots]&\none[\vdots] \\
			n_{l 1}  &n_{l 2} &n_{l3}&n_{l 4} &\none[\dots]          &    n_{l \lambda_j}    & \none & \none         & \none
		\end{ytableau}
		\caption{A standard Young tableau of the shape $\lambda$.}\label{figure}
	\end{figure}
	
	Denote the set of all standard Young tableaux of the shape $\lambda$ by $\Tab(\lambda)$. The number $c(p,q)=q-p$ is called the content  of the box with coordinates $(p,q)$ of the diagram $\lambda$ (see \cite{Macdonald}) or the content of the element $n_{pq}$ of tableau $T\in\Tab(\lambda)$ (see \cite{CST}).  We set $a_{i}=c(p,q)$ for $n_{pq}=i$. For each tableau $T$ define two functions $\mathfrak{R}_T$ and $\mathfrak{C}_T$ (row and column numbers) on the set $\{1,2,\ldots,N=|\lambda|\}$ as follows
	\begin{eqnarray}
		\mathfrak{R}_T(n_{pq})=p,\;\;\mathfrak{C}_T(n_{pq})=q.
	\end{eqnarray}

	\subsection{Realization of the irreducible representations of $\mathfrak{S}_{N}$.}\label{realisations}
	In this subsection we remind the explicit constructions of the irreducible representations of the finite group $\mathfrak{S}_{N}$ \cite{CST}, \cite{Okounkov-Vershik}. It is known that irreducible representations of $\mathfrak{S}_{N}$ are parameterized by the Young diagrams $\lambda$ with $N$ boxes. Let $\mathcal{R}_\lambda$ be one of those unitary irreducible representations which acts in the vector space $V_{\lambda}$. There exists an orthonormal basis $\left\{ v_T\right\}_{T\in\Tab(\lambda)}$ in $V_{\lambda}$ such that  the Coxeter generators $s_i=(i\;i+1)$ acts on elements of this basis  as follows:
	\begin{itemize}
		\item
		if $\mathfrak{R}_T(i)=\mathfrak{R}_T(i+1)$, then $\mathcal{R}_\lambda(s_i)v_T=v_T$;
		
		\item
		if $\mathfrak{C}_T(i)=\mathfrak{C}_T(i+1)$, then $\mathcal{R}_\lambda(s_i)v_T=-v_T$;
		
		\item
		if $\mathfrak{R}_T(i)\neq\mathfrak{R}_T(i+1)$ and $\mathfrak{C}_T(i)\neq\mathfrak{C}_T(i+1)$, then after the permutation of only two elements $i$ and $i+1$ in $T$ we obtain a tableau $T'$  again.
		Then the matrix of the operator $\mathcal{R}_\lambda(s_i)$ in basis $\{v_{T},v_{T'}\}$  in the two-dimensional space ${\rm span}\{v_{T},v_{T'}\}$  is
		\begin{equation}\label{matrix_generator}
			\begin{pmatrix}
				1/d & \sqrt{1-1/d^2}
				\\
				\sqrt{1-1/d^2} & -1/d
			\end{pmatrix},
		\end{equation}
		where $d=a_{i+1}-a_{i}$ and $a_i$ is the content of the box of $T$ that contain $i$.
	\end{itemize}
	Hence, we have the following statement.
	\begin{Lm}\label{sign_lemma}
		Let $\sigma\in\mathfrak{S}_N$. Put $\widehat{\mathcal{R}}_\lambda(\sigma)=\sgn(\sigma)\mathcal{R}_\lambda(\sigma)$. Then the representations $\widehat{\mathcal{R}}_\lambda$ and $\mathcal{R}_{\lambda'}$ of the group $\mathfrak{S_N}$  are unitary  equivalent.
	\end{Lm}
	Define a normalized  character of  the irreducible representation $\mathcal{R}_\lambda$ as follows
	\begin{equation}\label{character_of_finite_Gr} \chi_{_{\scriptstyle\lambda}}(\sigma)=\frac{\Tr\left( \mathcal{R}_\lambda(\sigma) \right)}{\Tr\left( \mathcal{R}_\lambda(\id) \right)}.
	\end{equation}
	where $\Tr$ is the standard trace of a finite-dimensional operator, $\sigma\in\mathfrak{S}_N$ and $\id$ is the identity element of $\mathfrak{S}_N$. Note that Lemma \ref{sign_lemma} implies that
	\begin{equation}\label{sign_char}
		\chi_{_{\scriptstyle\lambda'}}(\sigma)=\sgn(\sigma)\chi_{_{\scriptstyle\lambda}}(\sigma)
	\end{equation}
	for all $\sigma\in\mathfrak{S}_{N}$ and $\lambda\vdash N$.
	
	\begin{Lm}\label{matrix_elem_formula}
		Suppose that permutation $\sigma\in\mathfrak{S}_{N}$ is expressed as a product of distinct Coxeter generators:
		\begin{equation*}
			\sigma=s_{i_1}\ldots s_{i_r}~\text{where}~i_1<i_2<\ldots<i_r.
		\end{equation*}
		Then for any tableau $T$ the following equality holds
		\begin{equation}\label{matrix_elem_equ}
			|(\mathcal{R}_\lambda(\sigma) v_T,v_T)|=\prod_{j}\frac{1}{|a_{i_j+1}-a_{i_j}|},
		\end{equation}
		where the product is taken by all indices $j$ for which numbers $i_j$ and $i_j+1$ are in different rows and different columns of tableau $T$; i.e. $\mathfrak{R}_T(i_j)\neq\mathfrak{R}_T(i_j+1)$ and $\mathfrak{C}_T(i_j)\neq\mathfrak{C}_T(i_j+1)$.
	\end{Lm}
	\begin{proof}
		Consider an arbitrary Coxeter generator $s_i$ and a basis vector $v_T \in V_{\lambda}$, where $T\in\Tab(\lambda)$. If numbers $i$ and $i+1$ are in the same row or in the same column of $T$ we have $\mathcal{R}_\lambda(s_i)v_T=\pm v_T$. Thus,  in this case for any permutation $\sigma\in\mathfrak{S}_{N}$ we have
		\begin{equation}
			(\mathcal{R}_\lambda(\sigma s_i) v_T,v_T)=\pm (\mathcal{R}_\lambda(\sigma) v_T,v_T).
		\end{equation}
		Otherwise, according to \eqref{matrix_generator}, we have
		\begin{equation}
			\mathcal{R}_\lambda(s_i)v_T=\pm\frac{1}{a_{i+1}-a_{i}}v_{T}+\sqrt{1-\frac{1}{(a_{i+1}-a_{i})^2}}v_{T'},
		\end{equation}
		where $T'=s_i(T)$ is the tableau obtained from $T$ by permutation of numbers $i$ and $i+1$ (recall that $a_i$ and $a_{i+1}$ are the contents of the boxes of $T$ which contain $i$ and $i+1$ respectively).
		
		Next, we show that for any $\sigma\in\mathfrak{S}_{N}$ whose support $\supp_{_{N}}\sigma$ does not contain $i+1$ we have
		\begin{equation}
			(\mathcal{R}_\lambda(\sigma s_i)v_T,v_T)=\pm\frac{1}{a_{i+1}-a_{i}}(\mathcal{R}_\lambda(\sigma) v_T,v_T).
		\end{equation}
		Indeed, it suffices to check that the vector $\mathcal{R}_\lambda(\sigma) v_{T'}$ is orthogonal to $v_T$. It follows from the fact that $\mathcal{R}_\lambda(\sigma) v_{T'}$ is a linear combination of vectors $v_{T''}$ which correspond to those tableau $T''$ in which number $i+1$ is written in the same box as in $T'$. In particular, for all $T''$ we have $T''\neq T$ and hence, $(v_{T''},v_{T})=0$. Then $(\mathcal{R}_\lambda(\sigma) v_{T'},v_{T})=0$, so the required orthogonality is proved.
		
		Thus, for any permutation $\sigma\in\mathfrak{S}_{N}$ whose support does not contain $i+1$ we have
		\begin{equation*}
			(\mathcal{R}_\lambda(\sigma s_i)v_T,v_T)=\pm k_i(T)\cdot(\mathcal{R}_\lambda(\sigma) v_T,v_T),
		\end{equation*}
		where
		\begin{equation*}
			k_i(T)=
			\begin{cases}
				\frac{1}{a_{i+1}-a_{i}},~&\text{if}~\mathfrak{R}_T(i)\neq\mathfrak{R}_T(i+1)
				~\text{and}~\mathfrak{C}_T(i)\neq\mathfrak{C}_T(i+1),
				\\
				1,~&\text{otherwise}.
			\end{cases}
		\end{equation*}
		
		Now, let us apply this observation to $(\mathcal{R}_\lambda(\sigma )v_{T},v_T)$. Note that the sequence $\{i_j\}$ is strictly increasing, so for all $j$ the support of permutation $s_{i_1}\ldots s_{i_{j-1}}$ does not contain $i_j+1$. Therefore,
		\begin{eqnarray*}
			\begin{aligned}
				(\mathcal{R}_\lambda(\sigma) v_T,v_T)
				&=(\mathcal{R}_\lambda(s_{i_1}\ldots s_{i_r})v_T,v_T)=
				\\
				&=\pm k_{i_r}(T)(\mathcal{R}_\lambda(s_{i_1}\ldots s_{i_{r-1}})v_T,v_T)	=\ldots=\pm\prod_{j=1}^{r}k_{i_j}(T)(v_T,v_T).
			\end{aligned}
		\end{eqnarray*}
		Thus,
		\begin{equation*}
			|(\mathcal{R}_\lambda(\sigma) v_T,v_T)|=\prod_{j=1}^{r}|k_{i_j}(T)|
		\end{equation*}
		and we obtain the formula \eqref{matrix_elem_equ}. Lemma \ref{matrix_elem_formula} is proved.
	\end{proof}
	
	\subsection{Upper bound for the characters of the symmetric group}
	In the proof of the main theorem we use the following important bound for the characters of the symmetric group (see \cite{Roichman} and also \cite{Feray-Sniady}).
	
	\begin{Prop}[{\bf Roichman}, 1996]\label{char_bound}
		There exist absolute constants $a\in(0,1)$ and $b>0$ such that for any Young diagram $\lambda$ with $N=|\lambda|\ge 4$ boxes and for any $\sigma\in\mathfrak{S}_N$ the following inequality holds:
		\begin{equation}
			\left|\chi_{_{\scriptstyle\lambda}}(\sigma)\right|\leq
			\left(\max\left\{\frac{\lambda_1}{N},\frac{\lambda_1'}{N},a\right\}\right)^{b\cdot\#\left(\supp_{_N}(\sigma)\right)}.
		\end{equation}
		Here $\lambda_1$ ($\lambda_1'$) is the number of boxes in the first row (column) of the diagram $\lambda$.
	\end{Prop}
	
	\section{The approximation theorem for the characters on $\mathfrak{S}_{\widehat{\mathbf{n}}}$}
	In this section we state and prove the approximation theorem for characters on the group $\mathfrak{S}_{\widehat{\mathbf{n}}}$ which is a crucial part of the proof of Theorem \ref{main_theorem}.
	
	Firstly, let us recall the definition of a character.
	\begin{Def}
		A function $\chi\colon G\to\mathbb{C}$ on a group $G$ is called a \emph{character} if the following conditions hold
		\begin{itemize}
			\item[(a)]
			$\chi$ is \emph{central}, i.e. $\chi(gh)=\chi(hg)$ for all $g,h\in G$;
			
			\item[(b)]
			$\chi$ is a \emph{positive-definite function}, i.e. for any elements $g_1,\ldots,g_k\in G$ the matrix $\left[\chi(g_i g_j^{-1})\right]$ is a Hermitian and positive-semidefinite matrix;
			
			\item[(c)]
			$\chi$ is \emph{normalized}, i.e. $\chi(\id)=1$, where $\id$ is the identity of the group $G$.
		\end{itemize}
		If additionally
		\begin{itemize}
			\item[(d)]
			$\chi$ is \emph{indecomposable}, i.e. $\chi$ cannot be represented as a sum of two linear independent functions that satisfy (a) and (b),
		\end{itemize}
		then $\chi$ is called an \emph{indecomposable character}.
	\end{Def}
	\begin{Rem}
		Let $\pi_\chi$ be the representation of $G$, obtained from $\chi$ via the Gelfand-Naimark-Segal (shortly GNS) construction. Then the property ({\rm d}) is equivalent to $\pi_{\chi}$ being a factor representation.
	\end{Rem}
	
	The following fact is an analogue of the approximation theorem from \cite{Ver_Ker} for the characters of the standard infinite symmetric group $\mathfrak{S}_\infty$.
	\begin{Prop}\label{approx_thm}
6		Each indecomposable character $\chi$ on $\mathfrak{S}_{\widehat{\mathbf{n}}}$ is a weak limit of some sequence of normalized irreducible characters of the groups $\mathfrak{S}_{N_k}$. Namely, there is an increasing sequence $\left\{ k(l)\right\}_{l=1}^\infty$ of positive integers and there exist the partitions $\,^{k(l)}\!\lambda\vdash N_{k(l)}$ such that
		\begin{eqnarray*}
			\lim\limits_{l\to\infty} \chi_{_{^{k(l)}\!{\scriptstyle\lambda}}}(g)=\chi(g) \text{ for all } g\in\mathfrak{S}_{\widehat{\mathbf{n}}}.
		\end{eqnarray*}
	\end{Prop}
	\begin{proof}
		Consider the GNS-representation $\left( \pi_\chi,\mathcal{H}_\chi,\xi_\chi \right)$ of the group $\mathfrak{S}_{\widehat{\mathbf{n}}}$ acting in the Hilbert space $\mathcal{H}_\chi$ with the cyclic and separating vector $\xi_\chi$ such that $\left(\pi_\chi(g)\xi_\chi,\xi_\chi \right)=\chi(g)$ for all $g\in \mathfrak{S}_{\widehat{\mathbf{n}}}$. Denote by $\mathcal{B}(\mathcal{H}_\chi)$ the set of all bounded linear operators on $\mathcal{H}_\chi$. For any  conjugation-invariant  subset $\mathcal{S}\subset\mathcal{B}(\mathcal{H}_\chi)$ define its commutant by
		\begin{equation*}
			\mathcal{S}'=\left\{A\in\mathcal{B}(\mathcal{H}_\chi): AB=BA \text{ for all } B\in\mathcal{S}\right\},~\left( \mathcal{S}'\right)'=\mathcal{S}^{''}.
		\end{equation*}
		 We denote by $\left[\mathcal{S}\mathfrak{V}\right]$ the smallest closed subspace containing  $\mathcal{S}\mathfrak{V}$, where $\mathfrak{V}$ is a subset of $\mathcal{H}_\chi$.
		
		Since $\chi$ is a character, then according to the GNS-construction we have
		\begin{eqnarray}\label{bicyclic}
			\left[\pi_\chi(\mathfrak{S}_{\widehat{\mathbf{n}}})\xi_\chi\right]
			=\left[\pi_\chi(\mathfrak{S}_{\widehat{\mathbf{n}}})'\xi_\chi\right]=\mathcal{H}_\chi.
		\end{eqnarray}
		For convenience we denote the $w^*$-algebra $\left( \pi_\chi(\mathfrak{S}_{\widehat{\mathbf{n}}}) \right)^{''}$, which is a factor of type ${\rm II}_1$, as $M$. The faithful normal normalized trace $\tr$ on $M$ is a vector state defined by $\xi_\chi$, i.e. $\tr(a)=\left( a\xi_\chi,\xi_\chi \right)$ for $a\in M$.
		Denote by $L^2\left( M, \tr\right)$ the Hilbert space which is the completion of $M$ with respect to the norm which correspond to the inner product $<a,b>_{\tr}=\tr\,(b^*a)$, where $a,b\in M$.

From now on we suppose that $\mathcal{H}_\chi=L^2\left( M, \tr\right)$. Then we can assume that the operators  $a\in M$ acts by  {\it left multiplication}; i. e. $$L^2\left( M, \tr\right)\ni \eta\stackrel{a}{\rightarrow} a\eta\in L^2\left( M, \tr\right).$$
In this case the operators $a'\in M'$ are being  realized by operators of the {\it right multiplication} by the elements of $a\in M$
$$L^2\left( M, \tr\right)\ni \eta\stackrel{a'}{\rightarrow} \eta\;\cdot\;a\in L^2\left( M, \tr\right).$$
		
		Each element $g\in\mathfrak{S}_{\widehat{\mathbf{n}}}$ defines an inner automorphism ${\rm Ad}\,g$ of the factor $M$:
		\begin{eqnarray*}
			{\rm Ad}\,g (a)=\pi_\chi(g)a\pi_\chi(g^{-1}), \;\;a\in M.
		\end{eqnarray*}
		The map
		\begin{eqnarray}\label{def_cond_exp}
			M\ni a \stackrel{\,^{^{N_k}}\!\!E}{\mapsto} \frac{1}{N_k!}\sum\limits_{g\in\mathfrak{S}_{N_k}}{\rm Ad}\,g (a)\in M
		\end{eqnarray}
		is the {\it conditional expectation} \cite{Takesaki_1} which projects $M$ onto the subalgebra
		\begin{eqnarray*}
			\,^{N_k}\!\!M=\left\{a\in M: {\rm Ad}\,g (a)=a \;\text{ for all } g\in\mathfrak{S}_{N_k}\right\}.
		\end{eqnarray*}
		If $\lambda\vdash N_k$ and $P_\lambda=\frac{\dim^2\lambda}{N_k!}\sum\limits_{g\in \mathfrak{S}_{N_k}} \chi_\lambda(g)\pi_\chi(g)$, where ${\dim}\,\lambda$ is the dimension of the irreducible representation of $\mathfrak{S}_{N_k}$ that corresponds to $\lambda$, then $P_\lambda$ belongs to the center of the finite-dimensional algebra $\pi_\chi\left(\mathfrak{S}_{N_k}\right)^{''}$ and
		\begin{eqnarray}\label{decomposition}
			\,^{^{N_k}}\!\!E\left(\pi_\chi(g)\right)=\sum\limits_{\lambda\vdash N_k}\chi_\lambda(g)\,P_\lambda\;\text{ for all } \; g\in\mathfrak{S}_{N_k}.
		\end{eqnarray}
		The conditional expectations $\,^{^{N_k}}\!\!E$ are orthogonal projections in $L^2(M,\tr)$. Since $\,^{^{N_k}}\!\!E\geq \,^{^{N_{k+1}}}\!\!E$, we have
		\begin{eqnarray}\label{norm_convergence}
			\lim\limits_{k\to\infty}\left\|\,^{^{N_k}}\!\!E(a)-\,^\infty\!E(a) \right\|_{L^2(M,\tr)}=0 \;\text{ for any } a\in M,
		\end{eqnarray}
		where $\,^\infty\!E=\lim\limits_{k\to\infty}\,^{^{N_k}}\!\!E$.
		Since $w^*$-algebra $M=\pi_\chi\left( \mathfrak{S}_{\widehat{\mathbf{n}}}\right)^{''}$ is a factor, formula \eqref{def_cond_exp} implies that
		\begin{eqnarray*}
			\,^\infty\!E(a)=\tr(a)I \;\text{ for any } a\in M.
		\end{eqnarray*}
		Thus, combining \eqref{decomposition} and \eqref{norm_convergence}, we obtain
		\begin{eqnarray}\label{individual_average_estimate}
			\lim\limits_{k\to\infty} \sum\limits_{\lambda\vdash N_k}\left(\chi_\lambda(g)-\chi(g)\right)^2 \tr\left(P_\lambda\right)=0 \;\text{ for any } g\in\mathfrak{S}_{\widehat{\mathbf{n}}}.
		\end{eqnarray}
		Consider two sequences of positive reals $\left\{\epsilon_l \right\}$ и $\left\{\delta_l \right\}$ which satisfy the following conditions:
		\begin{eqnarray}\label{sequences}
			\lim\limits_{l\to\infty}\max\left\{\epsilon_l,\delta_l \right\}=0 \text{ и } \lim\limits_{l\to\infty}\frac{\epsilon_l\,N_l!}{\delta_l}=0.
		\end{eqnarray}
		Now, using \eqref{individual_average_estimate} we find  for each $l$ a positive integer $k(l)$ such that
		\begin{eqnarray}\label{prelimit_estimate}
			\sum\limits_{\lambda\vdash N_{m}}\left(\chi_\lambda(g)-\chi(g)\right)^2 \tr\left(P_\lambda\right)<\epsilon_l \text{ for all } g\in\mathfrak{S}_{N_l} \text{ and } m\geq k(l).
		\end{eqnarray}
		Put $\Lambda(m,g)=\left\{\lambda\vdash N_m: \left(\chi_\lambda(g)-\chi(g)\right)^2>\delta_l \right\}$, where $g\in\mathfrak{S}_{N_l}$, and $\,^k\!\Lambda(m)=\bigcup\limits_{g\in\mathfrak{S}_{N_l}}\Lambda(m,g)$. Applying \eqref{prelimit_estimate}, we obtain
		\begin{eqnarray}
			\sum\limits_{\lambda:(\lambda\vdash N_m)\&(\lambda\notin \,^l\!\Lambda(m))}\tr(P_\lambda)>1-\frac{\epsilon_l\,N_l!}{\delta_l}.
		\end{eqnarray}
		This inequality and \eqref{sequences} imply that for each $l$ there exists a partition $\,^{k(l)}\!\lambda\vdash N_{k(l)}$ such that
		\begin{eqnarray*}
			\left|\chi\!_{_{\,^{k(l)}\!\lambda}}(g)-\chi(g)\right|\leq\sqrt{\delta_l}\;\text{ for all }\;g\in\mathfrak{S}_{N_l}.
		\end{eqnarray*}
		Since $\lim\limits_{l\to\infty}\delta_l=0$ we have $\lim\limits_{l\to\infty}\chi\!_{_{\,^{k(l)}\!{\scriptstyle\lambda}}}(g)=\chi(g)$ for all $g\in\mathfrak{S}_{\widehat{\mathbf{n}}}$. Thus, the sequences $\{k(l)\}_{l=1}^{\infty}$ and $\{\,^{k(l)}\!\lambda\vdash N_{k(l)}\}_{l=1}^{\infty}$ satisfy the required conditions. Proposition \ref{approx_thm} is proved.
	\end{proof}
	
	\section{Technical lemmas}
	
	In this section we establish several technical facts which are used in the proof of Theorem \ref{main_theorem}.
	
	First, let us introduce some notations.
	Let $\,^n\Upsilon$ be the set of all partitions of a positive integer $n$. Denote by $(n)$ the partition which consists of only one part and which corresponds to the Young diagram with only one row. Put $(1^n)=(n)'$, i.e. $(1^n)$ is the partition which consists of $n$ parts equal 1 and which corresponds to the Young diagram with only one column. For an arbitrary partition $\mu$ denote
	$\,^n\Upsilon_\mu=\left\{\lambda\in\,^n\Upsilon: \lambda\setminus\left(\lambda_1\right)=\mu\right\}$ and  $\,^n\widehat{\Upsilon}_\mu=\left\{\lambda\in\,^n\Upsilon: \lambda\setminus\left( 1^{\lambda'_1}\right)=\mu\right\}$, where $\lambda=\left(\lambda_1, \ldots, \lambda_l \right)$. It is clear that
	$\,^n\Upsilon_\mu=\left\{\lambda\in\,^n\Upsilon: \lambda'\setminus\left( 1^{\lambda_1}\right)=\mu'\right\}$
	and
	$\,^n\widehat{\Upsilon}_\mu=\left\{\lambda\in\,^n\Upsilon: \lambda'\setminus\left( \lambda'_1\right)=\mu'\right\}$.
	
	Define in the set of all sequences $\left\{\,^k\!\lambda\vdash N_k\right\}_{k=1}^{\infty}$ of partitions the subsets $C_\infty$, $\widehat{C}_\infty$ $C_\mu$ and $\widehat{C}_\mu$, where $\mu$ is a partition and $|\mu|<\infty$, as follows:
	\begin{itemize}
		\item {($C_\infty$)} there is a subsequence $\left\{k_i \right\}$ such that
		$$
		\lim\limits_{i\to\infty}\left|  \,^{k_i}\!\lambda\setminus\left( \,^{k_i}\!\lambda_1\right)\right|=\infty;
		$$
		
		\item
		{($\widehat{C}_\infty$)} there is a subsequence $\left\{k_i \right\}$ such that
		$$
		\lim\limits_{i\to\infty}\left|  \,^{k_i}\!\lambda\setminus\left( 1^{\,^{k_i}\!\lambda_1'}\right)\right|=\infty;
		$$
		
		\item {($C_\mu$)} there is a subsequence $\left\{k_i \right\}$ such that
		$\,^{k_i}\!\lambda\in \,^{^{\left(\!N_{k_i}\!\right)}}\!\Upsilon_\mu$,
		where $\mu$ is independent of $i$;
		
		\item {($\widehat{C}_\mu$)} there is a subsequence $\left\{k_i \right\}$ such that
		$\,^{k_i}\!\lambda\in \,^{^{\left(\!N_{k_i}\!\right)}}\!\widehat{\Upsilon}_\mu$,
		where $\mu$ is independent of $i$.
	\end{itemize}
	\begin{Rem}\label{transpose_c_sets}
		Clearly, $\widehat{C}_{\mu}$ ($\widehat{C}_{\infty}$) is the set of all sequences which can be obtained by transposing the sequences $\left\{\,^k\!\lambda\vdash N_k\right\}$ in $C_{\mu}$ ($C_\infty$).
	\end{Rem}
	
	The following statement is trivial.
	\begin{Lm}\label{mu_cases}
		The union $C_\infty\cup \widehat{C}_\infty\cup\bigcup\limits_\mu C_\mu\cup \bigcup\limits_\mu\widehat{C}_\mu$ coincides with the set of all sequences of partitions $\left\{ \,^k\!\lambda\vdash N_k\right\}_{k=1}^\infty$.
	\end{Lm}
	
	Now consider an arbitrary sequence $\left\{ \,^k\!\lambda\vdash N_k\right\}_{k=1}^\infty$ of partitions. Let $\mathcal{R}_{(\!^{k}\!\lambda)}$ be the irreducible representation of $\mathfrak{S}_{N_k}$ which acts in the Hilbert space $V_{(\!^{k}\!\lambda)}$. Let $\left\{v_T \right\}_{T\in\Tab(\!^{k}\!\lambda)}$ be the orthonormal basis of $V_{(\!^{k}\!\lambda)}$, whose vectors are parameterized by the Young tableaux of the shape $\lambda$ (see Subsection \ref{realisations}). Then
	\begin{eqnarray}\label{trace_sum_matrix_elements}
		\chi_{_{(\!^{k}\!\lambda)}}(\sigma)=\frac{\Tr\left( \mathcal{R}_{(\!^{k}\!\lambda)}(\sigma) \right)}{\Tr\left( \mathcal{R}_{(\!^{k}\!\lambda)}(\id) \right)}=\frac{\sum\limits_{T\in\Tab(\!^{k}\!\lambda)} \left(\mathcal{R}_{(\!^{k}\!\lambda)}(\sigma)\,v_T,v_T\right)}{\Tr\left(\mathcal{R}_{(\!^{k}\!\lambda)}(\id) \right)}.
	\end{eqnarray}
	\begin{Rem}\label{cycle_type}
		Suppose that $\sigma$ belongs to the conjugacy class $\mathfrak{C}_{\,^k\!\mathfrak{m}}$ of the group  $\mathfrak{S}_{N_k}$, which consists of the permutations of the cycle type $\,^k\!\mathfrak{m}=(\,^k\!m_1, \,^k\!m_2,\ldots, \,^k\!m_l)$ (see Subsection \ref{minimal_element} and Remark \ref{convention_cycle_type}).
		Recall that $\mathfrak{i}_{k,j}$ $(j>k)$ is the embedding of $\mathfrak{S}_{N_k}$ into $\mathfrak{S}_{N_j}$ (see Section \ref{group_def} and also \eqref{embedding}). Then  $\mathfrak{i}_{k,j}\left(\mathfrak{C}_{\,^k\!\mathfrak{m}}\right)\subset \mathfrak{C}_{\,^j\!\mathfrak{m}}$, where
		\begin{equation*}
			\,^j\!\mathfrak{m}=\left(\,^j\!m_1,\,^j\!m_2,\ldots,\,^j\!m_l  \right)=\left(  \,^k\!m_1\frac{N_j}{N_k},  \,^k\!m_2\frac{N_j}{N_k},\ldots,  \,^k\!m_l\frac{N_j}{N_k} \right),\;\frac{N_j}{N_k}=\prod\limits_{i=k+1}^jn_i.
		\end{equation*}
		In particular, if $\sigma\in \mathfrak{S}_{N_k}$, then $\#\left(\supp_{_{N_j}}\,\mathfrak{i}_{k,j}(\sigma)\right)=\#\left(\supp_{_{N_k}}\,\sigma\right)\cdot \frac{N_j}{N_k}$ and $\,\kappa_{_{N_{j}}}(\sigma) =\frac{N_j}{N_k}\,\kappa_{_{N_{k}}}(\sigma)$ (see Proposition \ref{char_bound}). From now on we identify $g\in \mathfrak{S}_{N_k}$ with its image $\mathfrak{i}_{k,j}(g)\in \mathfrak{S}_{N_j}$ while taking into account the changes of the cycle type.
	\end{Rem}
	Recall that $\sgn_{N_k}$ is a one-dimensional representation of the group $\mathfrak{S}_{N_k}$ defined as
	\begin{equation*}
		\sgn_{N_k}(s)=\left\{
		\begin{array}{rl}
			-1,&\text{if}~s\in\mathfrak{S}_{N_k}~\text{is an odd permutation},
			\\
			1,&\text{if}~s\in\mathfrak{S}_{N_k}~\text{is an even permutation}.
		\end{array}\right.
	\end{equation*}
	The following statement is immediate.
	\begin{Lm}
		For any $s\in \mathfrak{S}_{\widehat{\mathbf{n}}}$ there exists a positive integer $N(s)$ such that $\sgn_{N_k}(s)=\sgn_{N_j}(s)$ for all $j,k>N(s)$.
	\end{Lm}
	Hence, there exists a one-dimensional representation $\sgn_\infty$ of the group $\mathfrak{S}_{\widehat{\mathbf{n}}}$ defined as follows
	\begin{eqnarray}\label{sign_repr}
		\sgn_\infty(s)=\lim_{k\to\infty}\sgn_{N_k}(s),\;\;s\in \mathfrak{S}_{\widehat{\mathbf{n}}}.
	\end{eqnarray}
	\begin{Rem}\label{sgn_inf_rem}
		If the sequence $\widehat{\mathbf{n}}=\left\{n_k \right\}_{k=1}^\infty$ contains infinitely many even numbers, then $\sgn_\infty(s)=1$ for all $s\in \mathfrak{S}_{\widehat{\mathbf{n}}}$.
	\end{Rem}
	\begin{Rem}
		Suppose that there are only finitely many even numbers in the sequence $\widehat{\mathbf{n}}=\left\{n_k \right\}_{k=1}^\infty$. Then one can choose two sequences $\left\{s_n \right\}$ and $\left\{\sigma_n \right\}$ in $\mathfrak{S}_{\widehat{\mathbf{n}}}$ such that $\sgn_\infty(s_n)=1$ and $\sgn_\infty(\sigma_n)=-1$ for all $n$ but which converge to the same automorphism $s\in\overline{\mathfrak{S}}_{\widehat{\mathbf{n}}}$ with respect to the metric $\rho$ (see \eqref{metr})
		\begin{eqnarray*}
			\lim\limits_{n\to\infty}s_n=\lim_{n\to\infty}\sigma_n=s.
		\end{eqnarray*}
		Therefore, $\sgn_\infty$ cannot be extended by continuity to the closure $\overline{\mathfrak{S}}_{\widehat{\mathbf{n}}}$ of the group $\mathfrak{S}_{\widehat{\mathbf{n}}}$ with respect to the metric $\rho$.
	\end{Rem}
	
	\begin{Lm}\label{inf_mu}
		Let $\chi$ be a character of the group $\mathfrak{S}_{\widehat{\mathbf{n}}}$ and let $\{\,^k\!\lambda\vdash N_k\}_{k=1}^{\infty}$ be a sequence of partitions such that
		$\lim\limits_{k\to\infty}\chi\!_{_{(\,^k\!{\scriptstyle\lambda})}}(g)=\chi(g)$ for all $g\in\mathfrak{S}_{\widehat{\mathbf{n}}}$. If $\{\,^k\!\lambda\}_{k=1}^{\infty}\in C_\infty\cup \widehat{C}_\infty$, then
		$\chi(g)=\left\{
		\begin{array}{rl}
			1, \text{ if } g=\id;\\
			0,\text{ if } g\neq \id.
		\end{array}
		\right.$
	\end{Lm}
	\begin{proof}
		Clearly, $\chi(\id)=1$. Take any element $g\neq\id$ of the group $\mathfrak{S}_{\widehat{\mathbf{n}}}$.
		Due to Lemma \ref{sign_lemma}, it is sufficient to consider the case when $\{\,^k\!\lambda\}_{k=1}^{\infty}\in C_\infty$.
		Let $\left\{k_i\right\}_{i=1}^{\infty}$ be a subsequence which satisfies the condition
		\begin{equation}\label{lemma_inf_cond}
			\lim\limits_{i\to\infty} \left| \,^{k_i}\!\lambda/\left( \,^{k_i}\!\lambda_1\right)\right|=\infty.
		\end{equation}
Suppose that $g\in\mathfrak{S}_{N_{k_1}}\subset \mathfrak{S}_{\widehat{\mathbf{n}}}$.

		It is sufficient to consider the case when $k_i=i$ for all $i$, i.e. when the sequence $\{k_i\}_{i=1}^{\infty}$ coincides with the sequence $\{k\}_{k=1}^{\infty}$ (the proof of the general case is analogous).
		Then, Proposition \ref{char_bound} implies that
		\begin{equation*}
			\left|\chi\!_{_{\left(\!\,^{^{k}}\!\scriptstyle{\lambda}\!\right)}}(g)\right|\leq
			\left(\max\left\{\frac{\,^{k}\!\lambda_1}{N_{k}},a\right\}\right)^{b\cdot\,\#\left(\supp_{_{N_{k}}}(g)\right)}, \text{ where } a\in (0,1), b>0.
		\end{equation*}
		Hence, after passing to the limit  $k\to\infty$ we obtain $k\to\infty$
		\begin{eqnarray*}
			\left|\chi(g)\right|\leq\limsup\limits_{k\to\infty}
			\left(\frac{\,^{k}\!\lambda_1}{N_{k}}\right)^{b\cdot\,\#\left(\supp_{_{N_{k}}}(g)\right)}=
			\limsup\limits_{k\to\infty}
			\left(\frac{\,^{k}\!\lambda_1}{N_{k}}\right)^{b N_k\cdot\,\#\left(\supp_{_{N_{1}}}(g)\right)/N_1}.
		\end{eqnarray*}
		Therefore,
		\begin{eqnarray*}
			\left|\chi(g)\right|\leq\limsup\limits_{k\to\infty}\left(1-\frac{N_{k}-\,^{k}\!\lambda_1}{N_{k}}\right)^{b N_k\cdot\,\#\left(\supp_{_{N_{1}}}(g)\right)/N_1}=
			\\
			=\limsup\limits_{k\to\infty}\; \exp\left(-\left| \,^{k}\!\lambda/\left( \,^{k}\!\lambda_1\right)\right|\cdot\frac{b}{N_{1}}\cdot\,\#\left(\supp_{_{N_{1}}}(g)\right)\right).
		\end{eqnarray*}
		Since $\supp_{_{N_{1}}}(g)\,> 1$, then, using \eqref{lemma_inf_cond},  we obtain that $\chi(g)=0$. This finishes the proof.
	\end{proof}
	
	\begin{Lm}\label{fin_mu}
		Let $\chi$ be a character of the group $\mathfrak{S}_{\widehat{\mathbf{n}}}$ and let $\{\,^k\!\lambda\vdash N_k\}_{k=1}^{\infty}$ be a sequence of partitions such that
		$\lim\limits_{j\to\infty}\chi\!_{_{(^{^{\scriptscriptstyle k}}\!\!{\scriptstyle\lambda)}}}(\sigma)=\chi(\sigma)$ for all $\sigma\in\mathfrak{S}_{\widehat{\mathbf{n}}}$. If $\{\,^k\!\lambda\vdash N_k\}_{k=1}^{\infty}\in C_\mu$ for some partition $\mu$, then $\chi(\sigma)=\chi_{\nat}(\sigma)^{|\mu|}$.
	\end{Lm}
	\begin{proof}
		As in the proof of Lemma \ref{inf_mu}, it suffices to consider the case when $\,^k\!\lambda\setminus\left( \,^k\!\lambda_1 \right)=\mu$ for all $k$. If $\mu$ is the empty partition, then $\chi\!_{_{(^{^{\scriptscriptstyle k}}\!\!{\scriptstyle\lambda)}}}(\sigma)=1$ for each $\sigma\in\mathfrak{S}_{N_k}$ and the statement is trivial. Hence, we can now suppose that
		\begin{eqnarray}
			|\mu|\geq 1.
		\end{eqnarray}
		Recall that $\Tab\left( \,^k\!\lambda \right)$ is the set of all (standard) Young tableaux of the shape $\,^k\!\lambda$ filled with numbers $1,2,\ldots,N_k$. Let $\mathcal{R}_{(^k\!\lambda)}$ be the irreducible representation of the group $\mathfrak{S}_{N_k}$ defined in Subsection \ref{realisations} and let $\chi\!_{_{(^{^{\scriptscriptstyle k}}\!\!{\scriptstyle\lambda)}}}$ be its normalized character (see \eqref{character_of_finite_Gr}).
		
		The hook length formula implies that
		\begin{eqnarray}\label{hook}
			\begin{aligned}
				\#\left(\Tab(\,^k\!\lambda)\right)=\dim\,^k\!\lambda=\frac{\dim\mu}{|\mu|!}\cdot\frac{N_k!}{(N_k-|\mu|-m)!
					\prod\limits_{i=1}^{m}(N_k-|\mu|-i+1+\mu'_i)}\\
				=\frac{\dim\mu}{|\mu|!}\cdot
				\frac{\prod\limits_{i=1}^{|\mu|+m}(N_k-|\mu|-m+i)}{\prod_{i=1}^{m}(N_k-|\mu|-i+1+\mu'_i)}.
			\end{aligned}
		\end{eqnarray}
		where $\mu'=\left( \mu'_1,\mu'_2,\ldots,\mu'_m \right)$ and $\dim\,^k\!\lambda$ is the dimension of representation $\mathcal{R}_{(^k\!\lambda)}$.
		
		Let $\mathbb{S}^{N_k}_{|\mu|}$ be the family of all $|\mu|$-element subsets $\{i_1<i_2<\ldots<i_{|\mu|}\}$ of the set $\{1,2,\ldots, N_k\}$. Clearly, $\#\mathbb{S}^{N_k}_{|\mu|}=\binom{N_k}{|\mu|}$. Denote by $S(T,\mu)$ the subset of those elements of $\{1,2,\ldots,N_k\}$ which are located in the boxes of the diagram $\,^k\!\lambda\setminus\left( \,^k\!\lambda_1 \right)=\mu$ in a tableau $T\in \Tab(\,^k\!\lambda)$.
		Since for each $T\in\Tab(\,^k\!\lambda)$ we have $\#S(T,\mu)=|\mu|$ we can regard $S(T,\mu)$ as an element of $\mathbb{S}^{N_k}_{|\mu|}$. Note that a tableau $T$ is defined uniquely by the filling of diagram $\,^k\!\lambda\setminus\left( \,^k\!\lambda_1 \right)=\mu$.
		
		Denote by $\widehat{\Tab}(\,^k\!\lambda)$ the set of all tableaux $T\in \Tab(\,^k\!\lambda)$ such that
		$S(T,\mu)$ consists only of subsets $\left\{i_1<i_2<\ldots<i_{|\mu|}\right\}$ which satisfy $i_{l+1}-i_l>1$ for all $l=1,2,\ldots,|\mu|-1$.
		
		Consider an arbitrary element $\sigma\in\mathfrak{S}_{\widehat{\mathbf{n}}}$. Suppose that $\sigma$ is an element of $\mathfrak{S}_{N_k}$ and belongs to the conjugacy class $\mathfrak{C}_{\,^k\!\mathfrak{m}}$ of the group $\mathfrak{S}_{N_k}$ that contains the permutations of the cycle type $\,^k\!\mathfrak{m}=(\,^k\!m_1, \,^k\!m_2,\ldots, \,^k\!m_l)$ (see Subsection \ref{minimal_element}). For any $j>k$ denote by $\sigma\!_{(^j\!\mathfrak{m})}$ the minimal element of the conjugacy class $\mathfrak{C}_{\,^j\!\mathfrak{m}}$ that contains $\mathfrak{i}_{k,j}(\sigma)$. In view of Remark \ref{cycle_type} there is a rational number $\alpha$ independent of $j$ such that
		\begin{eqnarray}\label{support_of_min_el}
			\supp_{_{N_j}}\,\sigma\!_{(^j\!\mathfrak{m})}=\left\{1,2,\ldots,\alpha N_j\right\}~\text{for all}~j>k.
		\end{eqnarray}
		
		Now take an arbitrary positive integer parameter $Q>|\mu|$ which will tend to infinity. Denote by $\Tab_Q(\,^j\!\lambda)$ the subset of those tableaux in $\Tab(\,^j\!\lambda)$ whose first $Q$ boxes of the first row are filled with numbers $1,2,\ldots,Q$. It is clear that
		\begin{eqnarray}\label{Q_hook}
			\begin{aligned}
				&\#\left(\widehat{\Tab}(\,^j\!\lambda)\cap\Tab_Q(\,^j\!\lambda)\right)\geq
				\\
				\geq &\dim\mu\cdot\frac{(N_j-Q)(N_j-Q-3)\cdots(N_j-Q-3(|\mu|-1))}{|\mu|!}.
			\end{aligned}
		\end{eqnarray}
		According to \eqref{trace_sum_matrix_elements}, we have
		\begin{eqnarray*}
			\chi\!_{_{(^{^{\scriptscriptstyle j}}\!\!{\scriptstyle\lambda})}}(\sigma)=\frac{\sum\limits_{T\in\Tab(\,^j\!\lambda)}
				\left(\mathcal{R}_{(^j\!\lambda)}
				\left(\sigma\!_{(^j\!\mathfrak{m})}\right)v_T,v_T\right)}{\#\left(\Tab(\,^j\!\lambda)\right)}
		\end{eqnarray*}
		It follows from  \eqref{hook} and \eqref{Q_hook}  that
		\begin{eqnarray*}
			\lim\limits_{j\to\infty}\frac{\#\left(\widehat{\Tab}(\,^j\!\lambda)\cap\Tab_Q(\,^j\!\lambda)\right)}{\#\left(\Tab(\,^j\!\lambda)\right)}=1.
		\end{eqnarray*}
		Thus, in order to compute the limit $\lim\limits_{j\to\infty}\chi\!_{_{(^{^{\scriptscriptstyle j}}\!\!{\scriptstyle\lambda})}}(\sigma)$ we need to estimate the matrix elements $(\mathcal{R}_{(^{^{\scriptscriptstyle j}}\!\!{\scriptstyle\lambda)}}\left(\sigma\!_{(^j\!\mathfrak{m})}\right) v_T,v_T)$ for $T\in\widehat{\Tab}(\,^j\!\lambda)\cap\Tab_Q(\,^j\!\lambda)$.
		
		Take a tableau $T\in\widehat{\Tab}(\,^j\!\lambda)\cap\Tab_Q(\,^j\!\lambda)$ such that at least one element $e$ from $\supp_{_{N_j}}\,\sigma\!_{(^j\!\mathfrak{m})}$ (see \eqref{support_of_min_el})  there is  in the diagram $\mu$. Then either $e$ or $e-1$ belongs to  $\left\{j_i\right\}_{i=1}^r$ from \eqref{minimal_el_property}.
		
		Let us first consider the case when the transposition $s_e=(e\;\;e+1)$ is contained in the decomposition $\sigma\!_{(^j\!\mathfrak{m})}=s_{j_1}s_{j_2}\cdots s_{j_r}$, where $j_i<j_{i+1}$ (see \eqref{minimal_el_property}). Then, according to the definitions of sets $\widehat{\Tab}(\,^j\!\lambda)$ and $\Tab_Q(\,^j\!\lambda)$, we have
		\begin{eqnarray}\label{e_conditions}
			e+1\in\left( \,^j\!\lambda_1 \right), \;\; a_{e+1}\geq Q.
		\end{eqnarray}
		In other words, the number $e+1$ is contained in the first row of the tableau $T$. Since $\sigma\!_{(^j\!\mathfrak{m})}=s_{j_1}s_{j_2}\cdots s_{j_r}$ satisfies the conditions of Lemma \ref{matrix_elem_formula} we have
		\begin{equation*}
			|(\mathcal{R}_{(^{^{\scriptscriptstyle j}}\!\!{\scriptstyle\lambda)}}\left( \sigma\!_{(^j\!\mathfrak{m})}\right) v_T,v_T)|=\prod_{l}\frac{1}{|a_{i_l+1}-a_{i_l}|},
		\end{equation*}
		where the product is taken over all indices $l$ such that elements $i_l$ and $i_l+1$ are contained in different rows and different columns of the tableau $T$ (recall that $a_i$ is the content of the box of $T$ that contains $i$). Since $e$ and $e+1$ are in different rows and columns of $T$, we have
		\begin{eqnarray}\label{matrix_el_unequlity}
			\left|(\mathcal{R}_{(^{^{\scriptscriptstyle j}}\!\!{\scriptstyle\lambda)}}\left( \sigma\!_{(^j\!\mathfrak{m})} \right) v_T,v_T)\right|\leq\frac{1}{|a_{e+1}-a_e|}.
		\end{eqnarray}
		According to our assumption, $e$ is contained in the diagram $\mu$. Therefore,
		\begin{eqnarray*}
			-|\mu|\leq a_e\leq |\mu|-2.
		\end{eqnarray*}
		Thus, using \eqref{e_conditions} and \eqref{matrix_el_unequlity} we obtain
		\begin{eqnarray}\label{estimate}
			|(\mathcal{R}_{(^{^{\scriptscriptstyle j}}\!\!{\scriptstyle\lambda)}}\left( \sigma\!_{(^j\!\mathfrak{m})} \right) v_T,v_T)|\leq\frac{1}{Q-|\mu|+2}.
		\end{eqnarray}
		
		If the transposition $s_{e-1}=(e-1\;\;e)$ is contained in the decomposition $\sigma\!_{(^k\!\mathfrak{m})}=s_{j_1}s_{j_2}\cdots s_{j_r}$, where $j_i<j_{i+1}$ (see \eqref{minimal_el_property}) then we can obtain the estimate \eqref{estimate} in a similar way.
		
		Now let us estimate the number of tableaux $T\in\widehat{\Tab}(\,^j\!\lambda)\cap\Tab_Q(\,^j\!\lambda)$, whose rows, starting from the second (i.e. rows of the diagram $\mu$), contain only elements of the set
		\begin{eqnarray*}
			\left\{1,2,\ldots,N_j \right\}\setminus\left(\supp_{_{N_j}}\,\sigma\!_{(^j\!\mathfrak{m})}  \right) =\left\{1+\alpha N_j, 2+\alpha N_j,\ldots,N_j \right\}~\text{ (see \eqref{support_of_min_el})}.
		\end{eqnarray*}
		In other words, it means that $\supp_{_{N_j}}\,\sigma\!_{(^j\!\mathfrak{m})}$ is contained in the first row of $T$.
		Denote the set of all tableaux (not necessarily from $\widehat{\Tab}(\,^j\!\lambda)\cap\Tab_Q(\,^j\!\lambda)$) that satisfy this property by $\Tab^\sigma_Q(\,^j\!\lambda)$. Note that for all sufficiently large $j$ the inequality $|\mu|<Q<\alpha N_j$ holds. Therefore,  applying \eqref{support_of_min_el}, we obtain
		\begin{eqnarray}\label{complementation_to_supp}
			\begin{aligned}
				&\Tab^\sigma_Q(\,^j\!\lambda)=\left\{T:S(T,\mu)\subset \left\{1+\alpha N_j, 2+\alpha N_j,\ldots,N_j \right\}\right\},
				\\
				&\#\Tab^\sigma_Q(\,^j\!\lambda)=\dim\mu\cdot\binom{N_j-\alpha N_j}{|\mu|}.
			\end{aligned}
		\end{eqnarray}
		Put
		\begin{eqnarray}\label{matrix_elements_sums}
			\begin{aligned}
				&\,^j\!\Sigma_0=\sum\limits_{T\in\Tab^\sigma_Q(\,^j\!\lambda)}(\mathcal{R}_{(^{^{\scriptscriptstyle j}}\!\!{\scriptstyle\lambda)}}\left(\sigma\!_{(^k\!\mathfrak{m})}\right) v_T,v_T),
				\\
				&\,^j\!\Sigma_1=\sum\limits_{T\in\left(\Tab(\,^j\!\lambda)\setminus\Tab^\sigma_Q(\,^j\!\lambda)\right)}(\mathcal{R}_{(^{^{\scriptscriptstyle j}}\!\!{\scriptstyle\lambda)}}\left( \sigma\!_{(^k\!\mathfrak{m})} \right)v_T,v_T).
			\end{aligned}
		\end{eqnarray}
		From the definition of $\mathcal{R}_{(^{^{\scriptscriptstyle j}}\!\!{\scriptstyle\lambda)}}$ (see Subsection \ref{realisations}) we have
		\begin{eqnarray*}
			\mathcal{R}_{(^{^{\scriptscriptstyle j}}\!\!{\scriptstyle\lambda)}}\left(\sigma\!_{(^k\!\mathfrak{m})}\right)v_T=v_T ~\text{ for all } T\in \Tab^\sigma_Q(\,^j\!\lambda).
		\end{eqnarray*}
		Hence, using \eqref{complementation_to_supp}) we obtain
		\begin{eqnarray}
			\,^j\!\Sigma_0=\dim\mu\cdot\binom{N_j-\alpha N_j}{|\mu|}.
		\end{eqnarray}
		In order to estimate $\,^j\!\Sigma_1$ consider two subsets
		\begin{eqnarray*}
			\,^j\!\Tab_{10}=\left(\Tab(\,^j\!\lambda)\setminus\Tab^\sigma_Q(\,^j\!\lambda)\right)\cap
			\left(\widehat{\Tab}(\,^j\!\lambda)\cap\Tab_Q(\,^j\!\lambda)\right),\\
			\,^j\!\Tab_{11}=\left(\Tab(\,^j\!\lambda)\setminus\Tab^\sigma_Q(\,^j\!\lambda)\right)\setminus\left(\widehat{\Tab}(\,^j\!\lambda)\cap\Tab_Q(\,^j\!\lambda)\right).
		\end{eqnarray*}
		Then $\Tab(\,^j\!\lambda)\setminus\Tab^\sigma_Q(\,^j\!\lambda)=\,^j\!\Tab_{10}\sqcup\,\,^j\!\Tab_{11}$ and hence
		\begin{eqnarray}
			\,^j\!\Sigma_1=\sum\limits_{T\in\,^j\!\Tab_{10}}(\mathcal{R}_{(^{^{\scriptscriptstyle j}}\!\!{\scriptstyle\lambda)}}\left(\sigma\!_{(^k\!\mathfrak{m})}\right) v_T,v_T)+\sum\limits_{T\in\,^j\!\Tab_{11}}(\mathcal{R}_{(^{^{\scriptscriptstyle j}}\!\!{\scriptstyle\lambda)}}\left(\sigma\!_{(^k\!\mathfrak{m})}\right) v_T,v_T).
		\end{eqnarray}
		Applying the bound \eqref{estimate} to matrix elements in the first sum we obtain
		\begin{eqnarray}\label{estimate_10}
			\left|\,^j\!\Sigma_1\right|\leq \frac{\#\,^j\!\Tab_{10}}{Q-|\mu|+2}+\#\,^j\!\Tab_{11}.
		\end{eqnarray}
		Next, combining \eqref{hook} and \eqref{Q_hook} gives
		\begin{eqnarray*}
			\#\,^j\!\Tab_{11}\leq \frac{\dim\mu}{|\mu|!}\left(
			\frac{\prod\limits_{i=1}^{|\mu|+m}(N_j-|\mu|-m+i)}{\prod_{i=1}^{m}(N_j-|\mu|-i+1+\mu'_i)}- \prod\limits_{i=1}^{|\mu|} (N_j-Q-3(|\mu|-i))\right).
		\end{eqnarray*}
		Therefore, there exists a postive constant $C$ independent of $j$ such that
		\begin{eqnarray*}
			\#\,^j\!\Tab_{11}\leq C\,N_j^{|\mu|-1}~ \text{ for all }~ j.
		\end{eqnarray*}
		Hence, applying formula \eqref{hook} and bound \eqref{estimate_10} we obtain the inequality
		\begin{eqnarray}\label{residual}
			\frac{\left|\,^j\!\Sigma_1\right|}{\#\Tab(\,^j\!\lambda)}\leq\frac{1}{Q-|\mu|+2}+\frac{C_1}{N_j},
		\end{eqnarray}
		where $C_1$ is s positive constant independent of $j$.
		
		Finally, let us estimate
		\begin{equation*}
			\left|\chi\!_{_{(^{^{\scriptscriptstyle j}}\!\!{\scriptstyle\lambda})}}(\sigma\!_{(^k\!\mathfrak{m})})-\frac{\,^j\!\Sigma_0}{\#\Tab(\,^j\!\lambda)}\right|=
			\left| \chi_{_{(\!^{j}\!\lambda)}}(\sigma\!_{(^k\!\mathfrak{m})})-\frac{\dim\mu}{\#\Tab(\,^j\!\lambda)}\cdot\binom{N_j-\alpha N_j}{|\mu|}\right|.
		\end{equation*}
		Combining \eqref{trace_sum_matrix_elements},  \eqref{matrix_elements_sums} and \eqref{residual} we obtain
		\begin{eqnarray*}
			\left| \chi\!_{_{(^{^{\scriptscriptstyle j}}\!\!{\scriptstyle\lambda})}}(\sigma\!_{(^j\!\mathfrak{m})})-\frac{\dim\mu}{\#\Tab(\,^j\!\lambda)}\cdot\binom{N_j-\alpha N_j}{|\mu|}\right|\leq\frac{\left|\,^j\!\Sigma_1\right|}{\#\Tab(\,^j\!\lambda)}\leq\frac{1}{Q-|\mu|+2}+\frac{C_1}{N_j}.
		\end{eqnarray*}
		Passing to the limit $j\to\infty$ gives us the inequality
		$$\left|\chi(\sigma)-(1-\alpha)^{|\mu|}\right|\leq \frac{1}{Q-|\mu|+2}.$$
		Since $Q$ can be chosen arbitrarily large we have $\chi(\sigma)=(1-\alpha)^{|\mu|}$. The statement of lemma \ref{fin_mu} now follows from \eqref{character_fundamental} and \eqref{support_of_min_el}.
	\end{proof}
	
	\begin{Co}\label{transpose_fin_mu}
		Let $\chi$ be a character of the group $\mathfrak{S}_{\widehat{\mathbf{n}}}$, and let $\{\,^k\!\lambda\vdash N_k\}_{k=1}^{\infty}$ be a sequence of partitions such that
		$\lim\limits_{j\to\infty}\chi_{_{^{^{\scriptscriptstyle k}}\!\!{\scriptstyle\lambda}}}(\sigma)=\chi(\sigma)$ for all $\sigma\in\mathfrak{S}_{\widehat{\mathbf{n}}}$. If $\{\,^k\!\lambda\vdash N_k\}_{k=1}^{\infty}\in \widehat{C}_\mu$ for some partition $\mu$, then $\chi(\sigma)=\sgn_{\infty}(\sigma)\cdot\chi_{\nat}(\sigma)^{|\mu|}$.
	\end{Co}\label{cor_fin_mu}
	\begin{proof}
		The statement follows directly from Lemma \ref{fin_mu}, Remark \ref{transpose_c_sets} and formula \eqref{sign_char}.
	\end{proof}
	
	Denote by $\exchar$ the set of functions on $\mathfrak{S}_{\widehat{\mathbf{n}}}$ which are claimed in Theorem \ref{main_theorem} to be indecomposable. Namely, we put
	\begin{eqnarray}\label{ex_char}	
\begin{split}
\exchar^+\left(\mathfrak{S}_{\widehat{\mathbf{n}}} \right)=\left\{\chi_{\nat}^{p}:p\in\mathbb{N}\cup\{0\}\right\},\\
\exchar^-\left(\mathfrak{S}_{\widehat{\mathbf{n}}} \right)=\left\{\sgn_{\infty}\cdot\chi_{\nat}^{p}:p\in\mathbb{N}\cup\{0\}\right\},\\
\exchar=\exchar^+\left(\mathfrak{S}_{\widehat{\mathbf{n}}} \right)\cup\exchar^-\left(\mathfrak{S}_{\widehat{\mathbf{n}}} \right)\cup\{\chi_{\nat}^{\infty}\}.
\end{split}	
\end{eqnarray}
	
	\begin{Lm}\label{char_check}
		$\exchar\left(\mathfrak{S}_{\widehat{\mathbf{n}}} \right)$ is a subset of the set of all characters on $\mathfrak{S}_{\widehat{\mathbf{n}}}$, and the indecomposable characters belongs to $\exchar$.
	\end{Lm}
	\begin{proof}
		It is clear that $\chi_{\nat}^{\infty}=\sgn_{\infty}\cdot\chi_{\nat}^{\infty}$ is a character on $\mathfrak{S}_{\widehat{\mathbf{n}}}$. It remains to check that $\chi_{\nat}^p$ and $\sgn_{\infty}\cdot\chi_{\nat}^p$ are characters on the group $\mathfrak{S}_{\widehat{\mathbf{n}}}$.
		
		Recall that $\mathcal{F}$ is a unitary representation of the group $\mathfrak{S}_{\widehat{\mathbf{n}}}$ which acts in the Hilbert space $\mathcal{H}=L^2\left( \mathbb{X}_{\widehat{\mathbf{n}}},\nu_{\widehat{\mathbf{n}}} \right)\otimes l^2(\mathbb{Z})$ (see Subsection \ref{factor_repr_constr}). Moreover, this representation satisfies the following property: for $\xi_0=\gimel\otimes\delta_0\in\mathcal{H}$ the equality
		\begin{equation*}
			\chi_{\nat}(\sigma)=(\mathcal{F}(\sigma)\xi_0,\xi_0)_{\mathcal{H}}      \;\; \text( see\ \  (\ref{natural_character}))
		\end{equation*}
		holds for all $\sigma\in\mathfrak{S}_{\widehat{\mathbf{n}}}$. Thus, $\chi_{\nat}$ is a character on $\mathfrak{S}_{\widehat{\mathbf{n}}}$.
		
		Now for any $p\in\mathbb{N}\cup\{0\}$  consider the unitary representation $\mathcal{F}^{\otimes p}$ acting on $\mathcal{H}^{\otimes p}$. It is clear that
		\begin{equation*}
			\chi_{\nat}^{p}(\sigma)=(\mathcal{F}^{\otimes p}(\sigma)\xi_0^{\otimes p},\xi_0^{\otimes p})_{\mathcal{H}^{\otimes p}}~\text{for all}~\sigma\in\mathfrak{S}_{\widehat{\mathbf{n}}}.
		\end{equation*}
		Since $\tau(B)=(B\xi_0^{\otimes p},\xi_0^{\otimes p})_{\mathcal{H}^{\otimes p}}$ is a vector state, $\chi_{\nat}^{p}$ is a character on $\mathfrak{S}_{\widehat{\mathbf{n}}}$. Similarly, $\sgn_{\infty}\cdot\,\mathcal{F}^{\otimes p}$ is also a unitary representation acting on $\mathcal{H}^{\otimes p}$ and the same argument implies that $\sgn_{\infty}\cdot\chi_{\nat}^p$ is also a character on $\mathfrak{S}_{\widehat{\mathbf{n}}}$.

If $\chi$ is indecomposable characters then, applying Proposition \ref{approx_thm}, Lemmas \ref{mu_cases},  \ref{inf_mu}, \ref{fin_mu} and Corollary \ref{transpose_fin_mu}, we obtain that $\chi$ belongs to $\exchar\left(\mathfrak{S}_{\widehat{\mathbf{n}}} \right)$.
	\end{proof}
	
	\begin{Lm}\label{irr_lemma}
Each character from  $\exchar\left(\mathfrak{S}_{\widehat{\mathbf{n}}} \right)$ is indecomposable.
	\end{Lm}
	
	\begin{proof}
It is clear that character $\chi_p^-=\sgn_{\infty}\cdot\chi_{\nat}^{p}$ is indecomposable if and only if character $\chi_p^+=\chi_{\nat}^{p}$ is indecomposable.

First we recall that, by Proposition \ref{factor}, character $\chi_{\rm nat}$ is indecomposable.

 {\bf First proof.} We suppose  the opposite; i.~e.   some character $\chi^+_m$ is not extreme point in the set of all normalized characters. Then there  exist  the numbers $\alpha_p^+,\alpha_p^-\in [0,1]$ with the property
\begin{eqnarray}\label{decomposition_1}
\begin{split}
\chi^+_m(\sigma)=\alpha_0^++\alpha_0^-\sgn_{\infty}(\sigma)+\sum\limits_{j=1}^{m-1}\alpha_j^+\chi_j^+(\sigma)+
\sum\limits_{j=1}^{m-1}\alpha_j^-\chi_j^-(\sigma)\\
+\sum\limits_{j=m+1}^{\infty}\alpha_j^+\chi_j^+(\sigma)+
\sum\limits_{j=m+1}^{\infty}\alpha_j^-\chi_j^-(\sigma)+ \alpha_\infty\chi_\infty(\sigma) ~ \text{ for all }~ \sigma\in \mathfrak{S}_{\widehat{\mathbf{n}}}.
\end{split}
\end{eqnarray}
Take the sequence $\left\{\sigma_n\right\}\subset\mathfrak{A}_{\widehat{\mathbf{n}}}$, satisfying the following conditions:
\begin{eqnarray*}
\sigma_n\neq {\rm id} ~\text{ for all }~ n ~\text{and}~ \lim\limits_{n\to\infty}\chi_{\rm nat}(\sigma_n)=0.
\end{eqnarray*}
Substituting $\sigma_n$ instead $\sigma$ into (\ref{decomposition_1}) and passing to the limit $n\to\infty$, we obtain
\begin{eqnarray}\label{Zero_nul}
\alpha_0^++\alpha_0^-=0 \Rightarrow \alpha_0^+=0 ~\text{ and }~ \alpha_0^-=0.
\end{eqnarray}
Put $\mathfrak{A}_{\widehat{\mathbf{n}}}=\left\{\sigma\in\mathfrak{S}_n:\sgn_{\infty}(\sigma)=1 \right\}$.
Since a set $\left\{\chi_{\rm nat}(\sigma)\right\}_{\sigma\in\mathfrak{A}_{\widehat{\mathbf{n}}}}$ is dense in $[0,1)$, it follows from (\ref{decomposition_1})  that
\begin{eqnarray*}
\gamma^m=\sum\limits_{j=1}^{m-1}\left(\alpha_j^++\alpha_j^-\right)\gamma^j+
\sum\limits_{j=m+1}^{\infty}\left(\alpha_j^++\alpha_j^-\right)\gamma^j ~ \text{ for all } \gamma\in [0,1).
\end{eqnarray*}
An easy computation shows that
\begin{eqnarray*}
\alpha_j^+=\alpha_j^-=0 ~ \text{ for all  naturale }~ j\neq m.
\end{eqnarray*}
Hence, using (\ref{decomposition_1}) and (\ref{Zero_nul}), we obtain that
\begin{eqnarray*}
\chi^+_m(\sigma)=\alpha_\infty\chi_\infty(\sigma) ~ \text{ for all }~ \sigma\in \mathfrak{S}_{\widehat{\mathbf{n}}}.
\end{eqnarray*}
Therefore, $\chi^+_m=\chi_\infty$. Since character $\chi_\infty$ is indecomposable as a regular cha\-rac\-ter of the ICC group (Proposition 7.9 \cite{Takesaki_1}), this contradicts the assumption that  $\chi^+_m$ is decomposable character. \qed

{\bf Second proof.} By above we can to suppose that $m<\infty$. If a character $\chi^+_m$ is not indecomposable then the corresponding GNS-representation $(\pi,\xi,\mathcal{H})$, where $\xi$ is unit cyclic vector such that $\chi^+_m(g)=\left< \pi(g)\xi,\xi \right>_\mathcal{H}$ for all $g\in\mathfrak{S}_{\widehat{\mathbf{n}}}$, is not factor-representation. Therefore, there are nonzero orthogonal projection $E$ in the center $C(M)$ of $w^*$-algebra $M$, generated by operators $\left\{\pi(s)\right\}_{s\in\mathfrak{S}_{\widehat{\mathbf{n}}}}$, and $\delta\in(0,1)$ such that
\begin{eqnarray}
\|E\xi\|^2=\delta>0.
\end{eqnarray}
Since $\xi$ is cyclic vector, there exist finite subset  $\left\{s_i\right\}_{i=1}^{K_p}\subset\mathfrak{S}_{N_p}$ and a collection  $\left\{ \vartheta_i \right\}$ of the complex numbers with the property
\begin{eqnarray}\label{E_approximation}
\left\|E\xi-\sum\limits_{i=1}^{K_p}\vartheta_i\pi(s_i)\xi\right\|<\epsilon.
\end{eqnarray}
Let $q_n, r_n$, $\Theta_n$, $\mathfrak{I}_n$, $\,^{r_n}_{q_n}\!\,\!\mathfrak{S}_{\widehat{\mathbf{n}}}$ be the objects are the same as in Corollary \ref{densety_orbit}. It follows from (\ref{chi_nat}) that
 \begin{eqnarray}\label{multi}
\chi^+_m\left(s_i\mathfrak{I}_n(s_j)\right) =\chi^+_m\left(s_i\right)\cdot\chi^+_m\left(\mathfrak{I}_n(s_j)\right)~ \text{ for all }~ i,j.
 \end{eqnarray}
 Since $\chi^+_m$ is continuous under the topology, defined on  $\mathfrak{S}_{\widehat{\mathbf{n}}}$ by metric $\rho$, we have from (\ref{E_approximation})
\begin{eqnarray}\label{I_E_approx}
\left\|(I-E)\xi-\left(I-\sum\limits_{i=1}^{K_p}\vartheta_i\pi(\mathfrak{I}_n(s_i))\right)\xi\right\|<\epsilon ~\text{ for all  } ~ n> K.
\end{eqnarray}
For simplicity of the notations we put $A=\sum\limits_{i=1}^{K_p}\vartheta_i\pi(s_i)$ and $A_n=I-\sum\limits_{i=1}^{K_p}\vartheta_i\pi(\mathfrak{I}_n(s_i))$.
Without loss of generality we suppose that $A=A^*$ and $A_n=A_n^*$. Now we obtain the following chain of inequalities
\begin{eqnarray*}
\left<E\xi,(I-E)\xi\right>\stackrel{(\ref{E_approximation})}{\geq} \left<A\xi,(I-E)\xi\right>-\epsilon\|(I-E)\xi\|\\
\stackrel{(\ref{I_E_approx})}{\geq}\left<A\xi,A_n\xi\right>-\epsilon\|A\xi\|-\epsilon\|(I-E)\xi\|\\
\geq\left<A\xi,A_n\xi\right>-\epsilon(\|E\xi\|+\epsilon)-\epsilon\|(I-E)\xi\|\geq \left<A\xi,A_n\xi\right>-2\epsilon-\epsilon^2.
  \end{eqnarray*}
  Hence, applying (\ref{multi}), we have
  \begin{eqnarray*}
0=\left<E\xi,(I-E)\xi\right>\geq \left<A\xi,\xi\right>\left<A_n\xi,\xi\right>-2\epsilon-\epsilon^2\\
\geq\left(\left<E\xi,\xi \right>-\epsilon \right)\left(\left<(I-E)\xi,\xi \right>-\epsilon \right)-2\epsilon-\epsilon^2\geq\delta(I-\delta)-3\epsilon.
  \end{eqnarray*}
  This inequality is false for  $\epsilon< \frac{\delta-\delta^2}{3}$.
	\end{proof}

	\section{The proof of Theorem \ref{main_theorem}}
	
	In this section we prove the main result of the present paper.
	
	\begin{proof}[Proof of Theorem \ref{main_theorem}]
		Suppose that $\chi$ is an indecomposable character on $\mathfrak{S}_{\widehat{\mathbf{n}}}$. According to Proposition \ref{approx_thm}, there exist a subsequence $\{k(l)\}_{l=1}^{\infty}$ and a sequence $\{\,^{k(l)}\!\lambda\vdash N_{k(l)}\}_{l=1}^{\infty}$ of partitions such that
		\begin{equation}
			\chi(g)=\lim_{l\to\infty}\chi_{_{^{k(l)}\!{\scriptstyle\lambda}}}(g)~\text{for any}~g\in\mathfrak{S}_{\widehat{\mathbf{n}}}.
		\end{equation}
		Lemma \ref{mu_cases} implies that three cases are possible:
		\begin{itemize}
			\item
			The sequence $\{\,^{k(l)}\!\lambda\vdash N_{k(l)}\}_{l=1}^{\infty}$ belongs to the union $C_\infty\cup \widehat{C}_\infty$. In this case Lemma \ref{inf_mu} implies that $\chi=\chi_{\nat}^{\infty}$.
			
			\item
			The sequence $\{\,^{k(l)}\!\lambda\vdash N_{k(l)}\}_{l=1}^{\infty}$ belongs to $C_{\mu}$ for some partition $\mu$. In this case Lemma \ref{fin_mu} implies that $\chi=\chi_{\nat}^{|\mu|}$.
			
			\item
			The sequence $\{\,^{k(l)}\!\lambda\vdash N_{k(l)}\}_{l=1}^{\infty}$ belongs to $\widehat{C}_{\mu}$ for some partition $\mu$. In this case Corollary \ref{cor_fin_mu} implies that $\chi=\sgn_\infty\cdot\chi_{\nat}^{|\mu|}$.
		\end{itemize}
		Thus, the character $\chi$ equals either $\chi_{\nat}^{p}$, or $\sgn_\infty\cdot\chi_{\nat}^{p}$ for some $p\in\mathbb{N}\cup\{0,\infty\}$, i.e. $\chi\in\exchar$ (see \eqref{ex_char}). In other words, we proved that the set of all indecomposable characters on $\mathfrak{S}_{\widehat{\mathbf{n}}}$ is a subset of $\exchar$. Finally, Lemma \ref{irr_lemma} implies that all these functions are indeed indecomposable characters on $\mathfrak{S}_{\widehat{\mathbf{n}}}$.
	\end{proof}
	
	\section{The proof of Theorem \ref{isomorphclass}}
	
	In this section we prove Theorem \ref{isomorphclass} about the isomorphism classes of groups $\mathfrak{S}_{\widehat{\mathbf{n}}}$.
	
	We need the following simple lemmas in the proof of Theorem \ref{isomorphclass}.
	
	\begin{Lm}\label{equiv_cond_lemma}
		Let $\widehat{\mathbf{n}}'=\{n_k'\}_{k=1}^{\infty}$ and $\widehat{\mathbf{n}}''=\{n_k''\}_{k=1}^{\infty}$, where $n_k',n_k''>1$ for all $k$, be the sequences of positive integers. Put $N_k'=\prod_{i=1}^{k}n_k'$ and $N_k''=\prod_{i=1}^{k}n_k''$. Denote
		\begin{eqnarray}
			\Div(\widehat{\mathbf{n}}')=\{N\in\mathbb{N}:N~\text{divides}~N_k'~\text{for some}~k\},
			\\
			\Div(\widehat{\mathbf{n}}'')=\{N\in\mathbb{N}:N~\text{divides}~N_k''~\text{for some}~k\}
		\end{eqnarray}
		Then, the following conditions are equivalent:
		\begin{enumerate}
			\item[(a)]
			for each prime number $p$ the following condition holds:
			\begin{equation}
				\lim_{k\to\infty}\deg_p(N_k')=\lim_{k\to\infty}\deg_p(N_k'');
			\end{equation}
			
			\item[(b)]
			for any element $N_i'$ of the sequence $\{N_k'\}_{k=1}^{\infty}$ there is an element $N_j''$ of the sequence $\{N_k''\}_{k=1}^{\infty}$ such that $N_i'$ divides $N_j''$ and vice versa.
			
			\item[(c)]
			$\Div(\widehat{\mathbf{n}}')=\Div(\widehat{\mathbf{n}}'')$.
			
		\end{enumerate}	
	\end{Lm}
	\begin{proof}
		The implications $(a)\Rightarrow (b)$ and $(b)\Rightarrow (c)$ are trivial.
		It remain to prove the implications $(b)\Rightarrow (a)$ and $(c)\Rightarrow (b)$.

{\bf The case  $(b)\Rightarrow (a)$}.

Let the condition $(a)$ does not holds for some prime $p$.

First we consider the case when $\lim\limits_{k\to\infty}\deg_p(N_k')<\infty$ and $\lim\limits_{k\to\infty}\deg_p(N_k'')<\infty$.
For the sake of definiteness, we will   assume that  there exist a positive integers $s$ and $D_p$  with the properties:
\begin{itemize}
  \item {\bf i)}  $\frac{N_k'}{p^{s}}=pL_k'+{\it r\!e\!m}_k'$ for all $k>D_p$, where ${\it r\!e\!m}_k'\in\mathds{N}$, ${\it r\!e\!m}_k'<p$ and $L_k'$  is nonnegative integer; i.~e. $\lim\limits_{k\to\infty}\deg_p(N_k')=s$;
  \item {\bf ii)} $p^{s+1}\big|N_k''$ for all $k>D_p$ or $N_k''=L_k''\,p^s+{\it r\!e\!m}_k''$ for all $k$, where $1\leq{\it r\!e\!m}_k''<p^s$ and $L_k''$ is nonnegative integer; i.~e. $\lim\limits_{k\to\infty}\deg_p(N_k'')\neq s$.
\end{itemize}
Hence, if $p^{s+1}\big|N_k''$ then, by property  {\bf i)}, $\frac{N_j'}{N_k''}=\frac{A_j'}{p\,A_k''}$, where $A_j'=\frac{N_j'}{p^s}\in\mathds{N}$, $A_k''=\frac{N_k''}{p^s}\in\mathds{N}$, $p$ $\nmid A_j'$ and $k>D_p$. Here the notation $l\nmid j$ means that `1Therefore, $N_k''$ not divides $N_j'$ for all  $j$.

Now we take $N_j'$ with $j>D_p$ and assume that holds second part  from property {\bf ii)}; i~e. $N_k''=L_k''\,p^s+{\it r\!e\!m}_k'' $  for all $k$. Then $\frac{N_k''}{N_j'}=\frac{L''_k\,+\frac{{\it r\!e\!m}_k''}{p^s}}{A_j'}$. Hence, since $1\leq{\it r\!e\!m}_k''<p^s$, then $N_j'$ not divides $N_k''$ for all $k$.

Now we assume that $\lim\limits_{k\to\infty}\deg_p(N_k')=\infty$ and $\lim\limits_{k\to\infty}\deg_p(N_k'')=s<\infty$.

Then there exists natural $D_p'$  such that  for all $k,j>D_p'$
 \begin{eqnarray}
 N_j'=p^{s+1}A_j',\; \frac{N_k''}{p^s}=L''_k\,p+{\it r\!e\!m}_k'', ~\text{where}~ A_j',L''_k,{\it r\!e\!m}_k''\in\mathds{N} ~ \text{and} ~ {\it r\!e\!m}_k''<p.\;\;\;
 \end{eqnarray}
 Therefore, if $j>D_p'$ then $N_j'$ not divides $N_k''$ for all $k$. This proves the implication $(b)\Rightarrow (a)$.

 At last we will prove that $(c)\Rightarrow (b)$. Suppose that there exists $j$ with the property: ${N_j'}\nmid {N_k''}$ for all $k$; ~i.~e. $N_j'\notin \Div(\widehat{\mathbf{n}}'')$. Since $N_j'\in\Div(\widehat{\mathbf{n}}')$, this contradicts the condition $(c)$.
	\end{proof}
	The following statement is the direct consequence of  Theorem \ref{main_theorem}.
	\begin{Lm}\label{max_char}
Let $\chi^0$ be a trivial character on $\mathfrak{S}_{\widehat{\mathbf{n}}} $. Put $CH(\mathfrak{S}_{\widehat{\mathbf{n}}} )=\exchar\left(\mathfrak{S}_{\widehat{\mathbf{n}}} \right)\setminus \{\chi^0,\, \sgn_{\infty}\}$.
		For any element $\sigma\in\mathfrak{S}_{\widehat{\mathbf{n}}}\setminus\{\id\}$ we have
$\chi_{\nat}(\sigma)=\max\limits_{\chi\in CH(\mathfrak{S}_{\widehat{\mathbf{n}}} ) }\left\{|\chi(\sigma)| \right\}$.
	\end{Lm}	
	\begin{proof}[\rm\bf {Proof of Proposition \ref{isomorphclass}}]
		Let  $\widehat{\mathbf{n}}'=\{n_k'\}_{k=1}^{\infty}$ and $\widehat{\mathbf{n}}''=\{n_k''\}_{k=1}^{\infty}$, where $n_k',n_k''>1$ for all $k$. Suppose that there exists an isomorphism  $\alpha\colon\mathfrak{S}_{\widehat{\mathbf{n}}'}\to\mathfrak{S}_{\widehat{\mathbf{n}}''}$ . Then the map
\begin{eqnarray}\label{bijection_between_extremal_characters}
		\chi\ni\exchar\left(\mathfrak{S}_{\widehat{\mathbf{n}}}'' \right)\stackrel{\alpha^*}{\mapsto}\chi\circ\alpha\in\exchar\left(\mathfrak{S}_{\widehat{\mathbf{n}}}'\right)
\end{eqnarray}
is a bijection of a set $\exchar\left(\mathfrak{S}_{\widehat{\mathbf{n}}}'' \right)$ onto $\exchar\left(\mathfrak{S}_{\widehat{\mathbf{n}}}'\right)$. By definition $\alpha^*$, we have
\begin{eqnarray}
\alpha^*\left(\exchar^+\left(\mathfrak{S}_{\widehat{\mathbf{n}}}'' \right)\right)=\exchar^+\left(\mathfrak{S}_{\widehat{\mathbf{n}}}'\right).
\end{eqnarray}
Let $\chi_{\nat}'$ and $\chi_{\nat}''$ be the natural characters on $\mathfrak{S}_{\widehat{\mathbf{n}}}'$ and $\mathfrak{S}_{\widehat{\mathbf{n}}}''$, respectively (see (\ref{chi_nat}), (\ref{natural_character})). Using the characterization of the natural character from lemma \ref{max_char}, we obtain
\begin{eqnarray}
\alpha*\left(\chi_{\nat}'' \right)=\chi_{\nat}''\circ\alpha=\chi_{\nat}'.
\end{eqnarray}
Thus, from (\ref{chi_nat}) it follows that
\begin{eqnarray*}
\chi_{\nat}''\left(\mathfrak{S}_{\widehat{\mathbf{n}}}'' \right)=\chi_{\nat}''\left(\alpha\left(\mathfrak{S}_{\widehat{\mathbf{n}}}'\right) \right)= \chi_{\nat}'\left(\mathfrak{S}_{\widehat{\mathbf{n}}}'\right)
=\left\{\frac{L}{N_k'}:k\in\mathbb{N}, 0\le L\le N_k'\right\}\\
=\left\{\frac{p}{q}:q\in\Div(\widehat{\mathbf{n}}'),~0\le p\le q,~\gcd(p,q)=1\right\}.
\end{eqnarray*}
Hence
\begin{eqnarray*}
\left\{\frac{p}{q}:q\in\Div(\widehat{\mathbf{n}}'),~0\le p\le q,~\gcd(p,q)=1\right\}\\
=\left\{\frac{p}{q}:q\in\Div(\widehat{\mathbf{n}}''),~0\le p\le q,~\gcd(p,q)=1\right\}.
\end{eqnarray*}
Therefore, $Div(\widehat{\mathbf{n}}')=Div(\widehat{\mathbf{n}}'')$. Finally, Lemma \ref{equiv_cond_lemma} implies that the condition \eqref{isomorph_condition} holds.
	\end{proof}
\section{Appendix}
Take the natural numbers $p,q,r$ such that $1<p<q<r$.
 Let $\mathbb{X}_{\widehat{\mathbf{n}}}=\prod\limits_{k=1}^\infty \mathbb{X}_{n_k}$. For $x=(x_1,x_2,\ldots)\in\mathbb{X}_{\widehat{\mathbf{n}}}$ we set $\,^r_q\!x=(x_{q+1},x_{q+2},\ldots,x_r)\in\prod\limits_{k=q+1}^r\mathbb{X}_{n_k}$. Each element $y\in \prod\limits_{k=q+1}^r\mathbb{X}_{n_k}$ defines a cylindric set
	\begin{eqnarray}\label{cylinder}
		\,^r_q\!\mathbb{A}_y=\left\{ x\in\mathbb{X}_{\widehat{\mathbf{n}}}:\,\,^r_q\!x=y \right\}\subset\mathbb{X}_{\widehat{\mathbf{n}}}.
	\end{eqnarray}
Put $\,^{^q}\!\!N_r=\frac{N_r}{N_q}$.

Let us introduce subgroup $\,^r_q\!\mathfrak{S}_{\widehat{\mathbf{n}}}\subset \mathfrak{S}_{N_r}$ by
\begin{eqnarray*}
\,^r_q\!\mathfrak{S}_{\widehat{\mathbf{n}}}=\left\{\sigma\in\mathfrak{S}_{\widehat{\mathbf{n}}}:\left(\sigma (x) \right)_i=x_i ~\text{ for all }~ x\in \mathbb{X}_{\widehat{\mathbf{n}}}~\text{and}~i \notin\left\{q+1,q+2,\ldots,r \right\} \right\}.
\end{eqnarray*}
It is clear that $\,^r_q\!\mathfrak{S}_{\widehat{\mathbf{n}}}$ is isomorphic to $\mathfrak{S}_{(\!\,^{^q}\!\!N_r\!)}$.
For large enough $r$ we find the nonnegative integer numbers $m$ and ${\it r\!e\!m}$ such that
\begin{eqnarray}\label{conditions}
\,^{^q}\!\!N_r=N_p\,m+{\it r\!e\!m},~\text{ where }~ \,^{^q}\!\!N_r>N_p~ \text{and}~  {\it r\!e\!m}<N_p.
\end{eqnarray}
Further we denote by $T$  automorphism $O^{N_q}$, where $O$ has been defined in chapter \ref{prelim_1}.

Set $\mathbf{0}=(0,0,\ldots,0,\ldots)\in \mathbb{X}_{\widehat{\mathbf{n}}}$. Then the subsets in the collection $\left\{\,^r_q\!\mathbb{A}_{\mathbf{y}_i} \right\}_{i=0}^{\,^{^q}\!\!N_r-1}$, where $\mathbf{y}_i=\,^r_q\!(T^i(\mathbf{0}))$, are pairwise disjunct and $\bigsqcup\limits_{i=0}^{\,^{^q}\!\!N_r-1}\,^r_q\!\mathbb{A}_{\mathbf{y}_i}=\mathbb{X}_{\widehat{\mathbf{n}}}$. For simplicity of notation, we denote   $\,^r_q\!\mathbb{A}_{\mathbf{y}_i}$ by $\mathbb{B}_i$ and
 define periodic automorphism $\,_q\!T\in \,^r_q\!\mathfrak{S}_{\widehat{\mathbf{n}}}$ as follows
\begin{eqnarray}
\,_q\!T(x)=\left\{
	\begin{array}{rl}
	T\,(x)=O^{N_q}(x),\text{ if } x\in\bigcup\limits_{j=0}^{^{^q}\!\!N_r-2}\mathbb{B}_j;\\
	T^{(1-\,^{^q}\!\!N_r)}x=O^{N_q(1-\,^{^q}\!\!N_r)}x,\text{ if } x\in \mathbb{B}_{(^{^q}\!\!N_r-1)}.
	\end{array}
	\right.
\end{eqnarray}
Put
\begin{eqnarray}\label{set E}
\mathbb{E}_k=\bigcup\limits_{i=km}^{(k+1)m-1}\,\mathbb{B}_i.
\end{eqnarray}
Hence we can conclude that
\begin{eqnarray}\label{measure_of_E_k}
\nu_{_{\widehat{\mathbf{n}}}}\left(\mathbb{E}_k\right)=\frac{m}{\,^{^q}\!\!N_r}=\frac{m}{N_p\,m+{\it r\!e\!m}}=\frac{1}{N_p}-\frac{\it r\!e\!m}{N_p\left(mN_p+{\it r\!e\!m} \right)}
\end{eqnarray}
for all $k\in\overline{0,N_p-1}$.

We now define automorphism $\,_q^{^p}\!T$ as follows
\begin{eqnarray}
\,_q^{^p}\!T\,x=\left\{
	\begin{array}{rl}
	\,_q\!T^m\,x	,\text{ if } x\in\bigcup\limits_{j=0}^{N_p-2}\mathbb{E}_j;\\
	\,_q\!T^{m(1-N_p)},\text{ if } x\in \mathbb{E}_{N_p-1};\\
     x~\text{ if } x\in\mathbb{X}_{\widehat{\mathbf{n}}}\setminus\left(\bigcup\limits_{j=0}^{N_p-1}\mathbb{E}_j \right).
	\end{array}
	\right.
\end{eqnarray}
Hence we obtain that $\,_q^{^p}\!T^{N_p}(x) =x$ for all $x\in\mathbb{X}_{\widehat{\mathbf{n}}}$ and
\begin{eqnarray}
\,_q^{^p}\!T\left(\mathbb{E}_k \right)=\mathbb{E}_{k+1({\rm mod}\,N_p)}
\end{eqnarray}
   We recall that
   \begin{eqnarray}\label{aut_beta}
   (\,_q^{^p}\!T)\,^p\!\mathbb{A}_y  =\,^p\!\mathbb{A}_y ~\text{ for all }~  y\in \prod\limits_{k=1}^p\mathbb{X}_{n_k}.
     \end{eqnarray}
     If $\mathbf{0}_p=\underbrace{(0,\ldots,0)}_{p}$, then the cylindric sets from the collection $\left\{\left({^{^p}_{\scriptscriptstyle{\mathbf{0}}}}\!O\right)^i \,^p\!\mathbb{A}_{\mathbf{0}_p}=\,^p\!\mathbb{A}_{\mathbf{y}_i}\right\}_{i=0}^{N_p-1}$, where $\mathbf{y}_i=\,^r_q\!(T^i(\mathbf{0}))$, are pairwise disjunct and $\bigsqcup\limits_{i=0}^{N_p-1}\,^p\!\mathbb{A}_{\mathbf{y}_i}=\mathbb{X}_{\widehat{\mathbf{n}}}$. Therefore, the following expression
\begin{eqnarray}\label{Theta_involution}
\Theta (x)=\left\{
	\begin{array}{rl}
	\left({^{^p}_{\scriptscriptstyle{\mathbf{0}}}}\!O\right)^{i-j}\,\left(\,_q^{^p}\!T \right)^{j-i} x,&~\text{ if }~ x\in \mathbb{E}_{i}\cap\,^p\!\mathbb{A}_{\mathbf{y}_j};\\
     x,&~\text{ if } x\in\mathbb{X}_{\widehat{\mathbf{n}}}\setminus\left(\bigcup\limits_{j=0}^{N_p-1}\mathbb{E}_j \right)
	\end{array}
	\right.
\end{eqnarray}
define automorphism from $\mathfrak{S}_{N_r}$. An ordinary verification shows that $\Theta^2x=x$ for all $x\in \mathbb{X}_{\widehat{\mathbf{n}}}$.

 Denote by $\,^\mathbb{E}\!\!\left[\,_q^{^p}\!T\right]$ the subgroup, consisting of the automorphisms $\alpha_r\in\,^r_q\!\mathfrak{S}_{\widehat{\mathbf{n}}}$, acting   as follows
 \begin{eqnarray}
 \alpha_r (x)=\left\{
	\begin{array}{rl}
	(\,_q^{^p}\!T)^{k_i} \,(x),&~\text{ if }~ x\in \mathbb{E}_{i};\\
     x,&~\text{ if } x\in\mathbb{X}_{\widehat{\mathbf{n}}}\setminus\left(\bigcup\limits_{j=0}^{N_p-1}\mathbb{E}_j \right),
	\end{array}
	\right.
 \end{eqnarray}
 where $\left(k_0,k_1,\ldots,k_{N_p-1} \right)$ is the collection of the integer numbers, belonging to the set $\left\{0,1,\ldots,N_p-1 \right\}$. Hence, using (\ref{Theta_involution}), we have
 \begin{eqnarray}\label{sigma_r_action}
 \sigma_r(x)=\Theta^{-1}\alpha_r\Theta\,(x)=\left\{
	\begin{array}{rl}
	\left({^{^p}_{\scriptscriptstyle{\mathbf{0}}}}\!O\right)^{k_i}\,(x),&~\text{ if }~ x\in \,^p\!\mathbb{A}_{\mathbf{y}_i}\cap\left(\bigcup\limits_{j=0}^{N_p-1}\mathbb{E}_j \right),\\
     x,&~\text{ if } x\in\mathbb{X}_{\widehat{\mathbf{n}}}\setminus\left(\bigcup\limits_{j=0}^{N_p-1}\mathbb{E}_j \right).
	\end{array}
	\right.
 \end{eqnarray}
 Now we take an arbitrary automorphism $\sigma\in \mathfrak{S}_{N_p}$.
 By lemma \ref{free_action}, for every $\sigma\in\mathfrak{S}_{N_p}$ exists a collection of the nonnegative integers  $\left\{k_i(\sigma) \right\}_{i=0}^{N_p-1}$, where $k_i(\sigma)\in\left\{0,1,\ldots,N_p-1 \right\}$, such that $\sigma(x)=\left({^{^p}_{\scriptscriptstyle{\mathbf{0}}}}\!O\right)^{k_i(\sigma)}\,(x)$ if $x\in\,^p\!\mathbb{A}_{\mathbf{y}_i}$. Since $q>p$, the cylindric sets $\,^r_q\!\mathbb{A}_{\mathbf{y}_i}$ are invariant under the group $\mathfrak{S}_{N_p}$;  i. e. $s\left( ^r_q\!\mathbb{A}_{\mathbf{y}_i}\right)=^r_q\!\mathbb{A}_{\mathbf{y}_i}$ for all $s\in\mathfrak{S}_{N_p}$ and $i=0,1,\ldots, \,^{^q}\!\!N_r-1$. Therefore, automorphism $\sigma_r\in \,\mathfrak{S}_{N_p}\,\cdot\,^r_q\!\mathfrak{S}_{\widehat{\mathbf{n}}}$ from (\ref{sigma_r_action}), where we will write $k_i(\sigma)$ instead $k_i$, is well defined.  If $\rho$ is the metric introduced in (\ref{metr}) then, using (\ref{conditions}), (\ref{measure_of_E_k})  and (\ref{sigma_r_action}), we have
 \begin{eqnarray}
 \rho\left(\sigma,\sigma_r \right)\leq\frac{\it r\!e\!m}{mN_p+{\it r\!e\!m}}\leq\frac{1}{m} , ~\text{ where }~ mN_p=\,^{^q}\!\!N_r-{\it r\!e\!m}, {\it r\!e\!m}<N_p.
 \end{eqnarray}
From the above, we obtain the following statement.
\begin{Prop}
Let the natural numbers $p,q,r$ satisfy the inequality $1<p<q<r$  and the condition (\ref{conditions}). There exist
injective homomorphism $\mathfrak{I}: \mathfrak{S}_{N_p}\rightarrow\,^r_q\!\mathfrak{S}_{\widehat{\mathbf{n}}}$
and automorphism $\Theta\in\,\mathfrak{S}_{N_p}\,\cdot\,^r_q\!\mathfrak{S}_{\widehat{\mathbf{n}}}$ such that
\begin{eqnarray}
\rho\left(\Theta\sigma\Theta^{-1},\mathfrak{I}(\sigma)\right)\leq\frac{1}{m}~\text{ for all }~\sigma\in\mathfrak{S}_{N_p}.
\end{eqnarray}
\end{Prop}
\begin{Co}\label{densety_orbit}
For a fixed $p>1$ take two sequences of the positive integer $q_n$,  $r_n$ such that $p\leq q_n<r_n$  and $\lim\limits_{n\to\infty}q_n=\lim\limits_{n\to\infty}(r_n-q_n)=\infty$. Let $\,^{^{(\!q_n\!)}}\!\!N_{r_n}=N_p\,m_n+{(\it r\!e\!m)_n}$, where $ m_n>1$ and  ${(\it r\!e\!m)_n}<N_p$. Then for each $n$ there exist injective homomorphism $\mathfrak{I}_n:\mathfrak{S}_{N_p}\rightarrow\,^{r_n}_{q_n}\!\,\!\mathfrak{S}_{\widehat{\mathbf{n}}}$ and automorphism $\Theta_n\in\mathfrak{S}_{N_p}\cdot\,^{r_n}_{q_n}\!\,\!\mathfrak{S}_{\widehat{\mathbf{n}}}$ such that
\begin{eqnarray}
\rho\left(\Theta_n\sigma\Theta_n,\mathfrak{I}_n(\sigma)\right)\leq\frac{1}{m_n} ~\text{for all }~ \sigma \in\mathfrak{S}_{N_p}.
\end{eqnarray}
Since $\lim\limits_{n\to\infty}(r_n-q_n)=\infty$, we have $\lim\limits_{n\to\infty}\rho\left(\Theta_n\sigma\Theta_n,\mathfrak{I}_n(\sigma)\right)=0$.
\end{Co}

\end{document}